\documentclass[11pt,a4paper]{amsart}
\usepackage[T1]{fontenc}
\usepackage[utf8]{inputenc}
\usepackage{amsmath,amssymb,amsthm,mathtools}
\usepackage{mathrsfs}
\usepackage[margin=1in]{geometry}
\usepackage[expansion=false]{microtype}
\usepackage{booktabs,array,longtable}
\usepackage[colorlinks=true,linkcolor=blue,citecolor=blue,urlcolor=blue]{hyperref}
\allowdisplaybreaks
\numberwithin{equation}{section}
\newtheorem{theorem}{Theorem}[section]
\newtheorem{lemma}[theorem]{Lemma}
\newtheorem{proposition}[theorem]{Proposition}
\newtheorem{corollary}[theorem]{Corollary}
\theoremstyle{definition}

\newtheorem{conjecture}[theorem]{Conjecture}
\newtheorem*{ac}{Acknowledgements}
\theoremstyle{remark}
\newtheorem{remark}[theorem]{Remark}
\DeclareMathOperator{\Tr}{Tr}
\DeclareMathOperator{\Ent}{Ent}
\DeclareMathOperator{\id}{id}
\DeclareMathOperator{\dom}{dom}

\newcommand{\RR}{\mathbb R}
\newcommand{\CC}{\mathbb C}
\newcommand{\ph}{\varphi}
\newcommand{\Energ}{\mathcal E}
\newcommand{\ip}[2]{\left\langle #1,#2\right\rangle}
\newcommand{\sech}{\operatorname{sech}}

\newcommand{\bE}{\mathbb E}
\newcommand{\cM}{\mathcal M}
\newcommand{\cN}{\mathcal N}
\newcommand{\cL}{\mathcal L}
\newcommand{\cI}{\mathcal I}

\newcommand{\cH}{\mathcal H}
\newcommand{\cF}{\mathcal F}

\newcommand{\cB}{\mathcal B}
\newcommand{\bj}{\mathbf j}
\newcommand{\Id}{\mathrm Id}
\newcommand{\norm}[1]{\left\lVert#1\right\rVert}
\newcommand{\cmlsi}{\mathrm{CMLSI}}
\newcommand{\cfm}{\mathrm{CFM}}

\title[Logarithmic Sobolev Inequalities]{Sharp Logarithmic Sobolev Inequalities for $q$-Araki-Woods Ornstein-Uhlenbeck Semigroups}
\author{Jinsong Wu}
\address{Jinsong Wu, Beijing Institute of Mathematical Sciences and Applications, Beijing, 101408, China}
\email{wjs@bimsa.cn}

\date{}

\begin{document}

\begin{abstract}
In this paper, we establish the sharp $2$-logarithmic Sobolev inequality for $q$-Araki– Woods Ornstein-Uhlenbeck semigroups with $-1<q<1$ and obtain hypercontractivity along the Gross curve with the optimal coefficient, and determine the exact hypercontractive time for $1<p\le2\le r<\infty$.
Reverse hypercontractivity holds in the two separate ranges $0<p<r<1$ and $-\infty<p<r<0$. 
An explicit counterexample rules out the proposed universal reverse curve across zero. 
We also obtain the complete modified logarithmic Sobolev inequality via the intertwining property.
\end{abstract}

\maketitle

\section{Introduction}

Hypercontractivity was first proved by Nelson \cite{Nel73} for classical Ornstein--Uhlenbeck semigroups.
Gross \cite{Gross75} proved a surprising equivalence between the logarithmic Sobolev inequality and hypercontractivity.
The general framework for noncommutative hypercontractivity was obtained by Olkiewicz and Zegarlinski \cite{OZ99}, who established the equivalence between hypercontractivity and logarithmic Sobolev inequalities.
Hypercontractivity and logarithmic Sobolev inequalities have been investigated for quantum Markov semigroups on finite-dimensional von Neumann algebras in \cite{King14, GaoWang26, WanWu26}.

In 1993, Carlen and Lieb \cite{CarLieb93} proved the sharp hypercontractivity for fermionic Ornstein--Uhlenbeck semigroups.
In 1997, Biane \cite{Bia97} proved the sharp hypercontractivity for \(q\)-Gaussian \cite{BS91, BS94} Ornstein--Uhlenbeck semigroups.
In 2011, Lee and Ricard proved hypercontractivity for \(q\)-Araki--Woods Ornstein--Uhlenbeck semigroups and showed that hypercontractivity can occur if and only if the generator of the deformation is bounded.
In 2021, Bikram and Mohanta\cite{BiMo} studied the hypercontractivity for the mixed \(q\)-Araki–Woods Ornstein-Uhlenbeck semigroups.

There are also fruitful results on hypercontractivity for infinite-dimensional von Neumann algebras, such as hypercontractivity for Gaussian quantum Markov semigroups \cite{BRK, CS08, SXZ26, LWW26}, for group von Neumann algebras \cite{JPPP17, XieZhang}, and for quantum groups \cite{Zhang}, among others.

In this paper, we obtain the sharp \(2\)-logarithmic Sobolev inequality for \(q\)-Araki--Woods Ornstein--Uhlenbeck semigroups.
We prove a hypercontractive bound along the Gross curve for every \(1<p<r<\infty\), with optimal coefficient, and determine the exact hypercontractive time when \(1<p\le2\le r<\infty\). 
The corresponding exact-time formula for both exponents strictly below or strictly above \(2\) remains Conjecture~\ref{con:hyper}.
A one-qubit inequality alone does not settle that conjecture.
In \cite{BDR}, Beigi, Datta, and Rouz\'e obtained quantum reverse hypercontractivity for GNS-symmetric quantum Markov semigroups on finite-dimensional von Neumann algebras.
The \(q\)-Araki--Woods algebras are non-tracial in general.
We obtain quantum reverse hypercontractivity for \(q\)-Araki--Woods Ornstein--Uhlenbeck semigroups separately for \(0<p<r<1\) and \(-\infty<p<r<0\).
In \cite{WirZhang}, Wirth and Zhang proves complete modified logarithmic Sobolev inequalities for mixed \(q\)-Gaussian Ornstein–Uhlenbeck semigroups.
Using the intertwining techniques in \cite{Wir25}, we obtain complete modified logarithmic Sobolev inequalities for \(q\)-Araki--Woods Ornstein--Uhlenbeck semigroups.

The paper is organized as follows.
In Section 2, we recall noncommutative $L^p$ spaces and $q$-Araki-Woods Ornstein-Uhlenbeck semigroups. 
In Section 3, we prove the sharp $2$-logarithmic Sobolev inequalities and hypercontractivity for $q$-Araki-Woods Ornstein-Uhlenbeck semigroups.
In Section 4, we obtain the complete modified logarithmic Sobolev inequality for $q$-Araki-Woods Ornstein-Uhlenbeck semigroups.
In Section 5, we obtain the reverse hypercontractivity for $q$-Araki-Woods Ornstein-Uhlenbeck semigroups.

\begin{ac}
Jinsong Wu was supported by grants from Beijing Institute of Mathematical Sciences and Applications.
Jinsong Wu was supported by NSFC (Grant no. 12371124). 
\end{ac}

\section{Preliminaries}

In this section, we shall briefly recall the noncommutative $L^p$ spaces, $q$-Araki-Woods algebra and its Ornstein-Uhlenbeck semigroup.

\subsection{$L^p$ spaces and quantum Markov semigroups}
Suppose that $\cM$ is a matrix algebra with a faithful normal state $\varphi$ and $\Tr$ is an ordinary unnormalized trace on $\cM$.
We denote by $D_\varphi$ the density operator associated to $\varphi$.
For each $X\in \cM$ and $a\in \RR$, we define
\[
 \bj_\varphi^a(X)=D_\varphi^{a/2} X D_\varphi^{a/2}.
\]
Then the norm $\norm{\cdot}_{p, \varphi}$ is given by $\|X\|_{p,\varphi}=\left\|\bj_\varphi^{1/p}(X)\right\|_p$.
Differentiating the weights gives
\[
 \frac{d}{dp} \bj_\varphi^{1/p}(X)
 =-\frac{1}{2p^2}\big((\log D_\varphi)\bj_\varphi^{1/p}(X)+\bj_\varphi^{1/p}(X)(\log D_\varphi)\big).
\]
The trace differentiation rule for a positive matrix then implies
\[
 \frac{d}{dp}\left\|\bj_\varphi^{1/p}(X)\right\|_p^p
 =\Tr\left(\bj_\varphi^{1/p}(X)^p\log \bj_\varphi^{1/p}(X)\right)-\frac1p\Tr\left(\bj_\varphi^{1/p}(X)^p\log D_\varphi\right).
\]
The entropy is 
\begin{align}\label{qubit:entropy}
 \Ent_{p,\varphi}(X)=\Tr\!\left(\bj_\varphi^{1/p}(X)^p(\log \bj_\varphi^{1/p}(X)^p-\log D_\varphi)\right)-\left\|\bj_\varphi^{1/p}(X)\right\|_p^p\log \left\|\bj_\varphi^{1/p}(X)\right\|_p^p.
\end{align}
By assuming that $\left\|\bj_\varphi^{1/p}(X)\right\|_p=1$, we have that 
\begin{equation}\label{qubit:rho}
 \Ent_{p,\varphi}(X)=D(\bj_\varphi^{1/p}(X)^p\Vert D_\varphi).
\end{equation}
where $D(D_1\|D_2)=\Tr(D_1(\log D_1-\log D_2))$ is the relative entropy.
By differentiating the logarithm of the norm, we obtain that 
\begin{equation}\label{qubit:p-norm-derivative}
 \frac{d}{dp}\log\|X\|_{p,\varphi} =
 \frac{\Ent_{p,\varphi}(X)}{p^2\|X\|_{p,\varphi}^p}.
\end{equation}
For \(X>0\), set
\[
 I_{p,2}(X):=D_\varphi^{-1/(2p)}(D_\varphi^{1/4} X D_\varphi^{1/4})^{2/p}D_\varphi^{-1/(2p)}=\bj_{\varphi}^{-1/p}(\bj_\varphi^{1/2}(X)^{2/p}).
\]
The entropy is preserved by this power map,
\[
 \Ent_{p,\varphi}(I_{p,2}(X))=\Ent_{2,\varphi}(X).
\]


A quantum Markov semigroup $\{\Phi_t\}_{t\geq 0}$ on a von Neumann algebra $\cM$ is a weak-$^*$ continuous semigroup of unital normal completely positive maps on $\cM$.
Let $\cL$ be its generator and $\Phi_t=e^{-t\cL}$.
Suppose that $\{\Phi_t\}_{t\geq 0}$ is a quantum Markov semigroup on a finite dimensional von Neumann algebra with an invariant state $\varphi$.
The Dirichlet form is 
\begin{align}\label{qubit:energy}
 \mathcal E_{p,\varphi}(X)
 &=
 \frac{pp'}{4}
 \Tr\!\left((\bj_\varphi^{1/p}(X))^{p-1}\bj_\varphi^{1/p}(\cL(X))\right).
\end{align}
Thus a \(p\)-logarithmic Sobolev constant \(\alpha_p\) is the optimal constant in the following inequality:
\begin{equation}
\label{qubit:lsi-convention}
 \alpha_p\Ent_{p,\varphi}(X)\leq \mathcal E_{p,\varphi}(X)
 \qquad (X>0).
\end{equation}
These conventions agree with the Dirichlet-form normalization used in \cite{BDR}.
Both sides of \eqref{qubit:lsi-convention} are homogeneous of degree \(p\) in \(X\). 
Moreover, we have that 
\begin{equation}\label{qubit:time-norm-derivative}
 \frac{d}{dt}\log\|\Phi_t(X)\|_{p,\varphi}
 =
 -\frac{4(p-1)}{p^2\|\Phi_t(X)\|_{p,\varphi}^p}
       \mathcal E_{p,\varphi}(\Phi_t(X)).
\end{equation}

If $\cM$ is a von Neumann algebra with a faithful normal state $\varphi$, we shall use the Haagerup $L^p$ space $L^p(\cM)$ for $0< p <\infty$ and Haagerup reduction for this space. 
For a faithful normal state $\varphi$ on $\cM$, let $D_\varphi$ be its Haagerup density, i.e. $tr(D_\varphi)=\varphi(1)$, where $tr$ is the $tr$-functional on $L^1(\cM)$.
Denote by $\Delta_\varphi$ the modular operator associated to $\varphi$ and $\sigma_t^{\varphi}$ the modular automorphism and $\Omega_\varphi$ the separating and cyclic vector.
See \cite{HJX} for more details on this material.
We shall use similar notations as in the matrix-algebra case:
\begin{equation}\label{main:norm}
 \bj_{\varphi}^{1/p}(x)=D_\varphi^{1/(2p)}xD_\varphi^{1/(2p)},\qquad
 \norm x_{p,\varphi}=\norm{\bj_{\varphi}^{1/p}(x)}_{L^p(\cM)}.
\end{equation}
This is a norm for $p\ge1$ and a quasi-norm for $0<p<1$, which is called symmetric (quasi-) norm of $x$ in the following.
The notation $\norm{\Phi_t}_{p\to r}\le1$ in this paper means $\norm{\Phi_t (x)}_{r,\varphi}\le\norm x_{p,\varphi}$ on bounded inputs $x\in \cM$.

An element $x\in \cM$ is \emph{entire modular-analytic} if
$t\mapsto\sigma_t^\varphi(x)$ extends to an entire operator-norm
analytic function on $\mathbb C$.
It has \emph{compact modular spectral support} if for some $R<\infty$,
\[
 \mathbf1_{[-R,R]}(\log\Delta_\varphi)x\Omega_\varphi=x\Omega_\varphi.
\]
This is equivalent to compact Arveson spectrum for $\sigma^\varphi$:
the identity $\sigma_t^\varphi(x)\Omega_\varphi
=\Delta_\varphi^{it}x\Omega_\varphi$ and the separating property of
$\Omega_\varphi$ identify their annihilating convolution functions.
Such an element is entire modular-analytic. Its symmetric vector
$\bj_{\varphi}^{1/2}(x)=\Delta_\varphi^{1/4}x\Omega_\varphi$ has the
same compact spectral support.

For $x\ge0$ and $x\neq 0$, the element
$\displaystyle D_{\rho_x}=\frac{1}{\norm x_{2,\varphi}^2}(D_\varphi^{1/4}xD_\varphi^{1/4})^2$ in $L^1(\cM)$, defines a normal state $\rho_x$. 
We use
\begin{equation}\label{eq:generalent}
 \Ent_{2,\varphi}(x)=\norm x_{2,\varphi}^2 D(\rho_x\Vert\varphi),
 \qquad \mathcal E_{2,\varphi}(x)=\norm{\cL^{1/2}\bj_{\varphi}^{1/2}(x)}_2^2.
\end{equation}
Relative entropy is Araki relative entropy and may be infinite.
The energy is defined when $\bj_{\varphi}^{1/2}(x)\in\dom \cL^{1/2}$;
otherwise it is $+\infty$. Set $\Ent_{2,\varphi}(0)=0$.
The second expression uses the $L^2$ implementation of $\cL$ for KMS-symmetric quantum Markov semigroup $\{\Phi_t\}_{t\geq 0}$. 
We define the norm $\norm{\cdot}_{\cL}$ as follows
\begin{align*}
    \norm{y}_{\cL}=\sqrt{\|y\|_{2}^2 +\norm{\cL^{1/2}(y)}_2^2}, \quad y\in \dom(\cL^{1/2}).
\end{align*}

Let $D_{\alpha,z}(\cdot \|\cdot)$ denote the $\alpha$-$z$ R\'enyi divergence of normal states, defined using Haagerup powers and their factorization, as in~\cite{HJ}. 
With $p\neq 2$, $\alpha=p/2$ and $z=p$, its defining expression gives the exact identity
\begin{equation}\label{eq:renyiidentity}
 \log\norm x_{p,\varphi}-\log\norm x_{2,\varphi}
 =\frac{p-2}{2p}D_{p/2,p}(\rho_x\Vert\varphi).
\end{equation}
For $2<p<3$, monotonicity in the second parameter gives
\[
 D_{p/2,3}(\rho_x\Vert\varphi)
 \le D_{p/2,p}(\rho_x\Vert\varphi)
 \le D_{p/2,2}(\rho_x\Vert\varphi).
\]
The following bound supplies the finiteness near $\alpha=1$. 
In fact,
\[
 Q_{2,2}(\rho_x\Vert\varphi)
 =\|x\|_{2, \varphi}^{-4}\|xD_\varphi^{1/2}x\|_2^2
 \le \|x\|_{2, \varphi}^{-4}\|x\|^4<\infty.
\]
Here $Q_{\alpha,z}=\exp((\alpha-1)D_{\alpha,z})$ for states.
\cite[Theorem 6.6]{HJ} implies the finiteness of the R\'enyi divergence.
Monotonicity \cite[Theorem 5.7]{HJ} in $z$ also gives $D_{2,3}(\rho_x\Vert\varphi)\le D_{2,2}(\rho_x\Vert\varphi) <\infty$.
The fixed-$z$ limit theorem~\cite[Theorem 6.9]{HJ} shows that both outer quantities converge to $D(\rho_x\Vert\varphi)$. 
It follows that
\begin{equation}\label{eq:typeIIIderivative}
 \left.\frac d{dp}\right|_{2^+}\log\norm x_{p,\varphi}
 =\frac14D(\rho_x\Vert\varphi)
 =\frac{\Ent_{2,\varphi}(x)}{4\norm x_{2,\varphi}^2}.
\end{equation}

\subsection{$q$-Araki-Woods Ornstein-Uhlenbeck semigroups}
Suppose that $\cH_{\RR}$ is a separable real Hilbert space and $\{U_t\}_{t\in \RR}$ is a strongly continuous one-parameter group of orthogonal transformations on $\cH_{\RR}$.
Let $\cH_{\CC}=\cH_{\RR}+i\cH_{\RR}$ be the complexification of $\cH_{\RR}$.
The group $\{U_t\}_{t\in \RR}$ extends to a group of unitary operators on $\cH_{\CC}$ and there exists a self-adjoint strictly positive operator $A$ such that $U_t=A^{it}$ and $J_{c} AJ_{c} =A^{-1}$, where $J_{c}$ is the real conjugation on $\cH_{\CC}$.
The operator $A$ is almost-periodic if it has an orthonormal basis of eigenvectors.
The inner product $\langle \cdot, \cdot\rangle_U$ with respect to $\{U_t\}_{t\in \RR}$ is defined as 
\begin{align*}
    \langle \xi, \eta\rangle_U=\langle 2A(1+A)^{-1} \xi, \eta\rangle_{\cH_{\CC}}, \quad \xi, \eta\in \cH_{\CC}.
\end{align*}
The Hilbert space completion with respect to the inner product is denoted by $\cH$.

The $q$-inner product $\langle \cdot, \cdot\rangle_q$ on the space $\cH^{\otimes n}$ is defined as follows
\begin{align*}
    \langle \xi_1\otimes \cdots \otimes \xi_n, \eta_1\otimes \cdots \otimes \eta_n\rangle_q
    = \sum_{\pi\in \Pi_n} q^{i(\pi)}\langle \xi_1, \eta_{\pi(1)}\rangle \cdots \langle \xi_n, \eta_{\pi(n)}\rangle,
\end{align*}
where $\Pi_n$ is the permutation group on $\{1, \ldots, n\}$, and
\begin{align*}
    i(\pi)=\#\{ (i,j)| 1\leq i<j \leq n, \, \pi(i) > \pi(j)\}
\end{align*}
is the number of inversions of $\pi\in \Pi_n$.
The Hilbert space with the $q$-inner product is denoted by $\cH^{\otimes_q n}$.
The $q$-Fock space $\cF_q(\cH)$ is
\begin{align*}
    \cF_q(\cH)=\CC\Omega \oplus \bigoplus_{n \geq 1} \cH^{\otimes _q n},
\end{align*}
where $\Omega$ is the vacuum vector.

For each $\xi\in \cH_{\RR}$, we define a generalized $q$-semi-circular random variable on $\cF_q(\cH)$ by
\begin{align*}
    s_q(\xi)=\ell_q(\xi)+ \ell^*_q(\xi),
\end{align*}
where $\ell_q(\xi)$ is the left creation operator by $\xi$ and $\ell_q^*(\xi)$ is its adjoint.
The $q$-Araki-Woods algebra $\Gamma_q(\cH_{\RR}, \{U_t\}_{t\in \RR})(=: \Gamma_q)$ is $\{s_q(\xi): \xi\in \cH_{\RR}\}''$.
The vacuum state $\varphi(\cdot)=\langle \cdot \Omega, \Omega\rangle_q$ is a normal faithful state on the $q$-Araki-Woods algebra $\Gamma_q$.
For $\xi_1, \ldots,\xi_m\in \cH_{\RR}$, we denote by $W(\xi_1\otimes \cdots \otimes \xi_m)$ the Wick operator in $\Gamma_q$.

The $q$-Ornstein-Uhlenbeck semigroup \cite{BKS97, Hia03} $\{\Phi_t\}_{t\geq 0}$ on $\Gamma_q$ is defined as
\begin{align*}
    \Phi_t(x)\Omega=\cF(e^{-t}\Id_{\cH})x\Omega,
\end{align*}
where $\cF(e^{-t}\Id_{\cH})$ on $\cF_q(\cH)$ is defined by $\cF(e^{-t}\Id_{\cH})|_{\cH^{\otimes n}}=e^{-nt} \Id_{\cH^{\otimes n}}$.

Suppose that $A$ is almost periodic and $\{\mu_j^4\}_{j \geq 1}\subseteq [1, \infty)$ is the sequence of representatives of eigenvalues $\{\mu_j^4, \mu_j^{-4}\}$ of $A$.
Suppose that $\{ e_{\pm k}\}_{k \geq 1}$ is an orthonormal basis of $\cH$ and $g_{q, i}=\mu_i^{-1}\ell_q(e_i)+\mu_i \ell_q^*(e_{-i})$ and $\sigma_t^{\varphi}(g_{q, i})=\mu_i^{4it}g_{q,i}$.
Denote by $\Gamma_q^{\mu}$ the von Neumann algebra generated by $\{g_{q, k}\}_{k \geq 1}$ and by $\varphi$ its vacuum state again.
Then there exists a spatial isomorphism between $(\Gamma_q, \varphi, \Omega)$ and $(\Gamma_q^{\mu}, \varphi, \Omega)$.
The hypercontractivity is stable under taking ultraproducts.
We shall only focus on the almost periodic case due to the discretization procedure in Section 6.1 of \cite{Nou06}.

Suppose that $\displaystyle x=\sum_{k=0}^\infty x_k$ is the orthogonal chaos decomposition, where $\cL(x_k)=k x_k$.
Then $\displaystyle \|\bj_\varphi^{1/2}(x)\|_{\cL}^2=\sum_{k=0}^\infty(1+k)\|x_k\|_{2,\varphi}^2$.

Let \(\beta \ge0\) and define
\begin{equation}\label{main:B}
 B_p(\beta)
 =\frac{\sinh(\beta/(2p'))}{\sinh(\beta/(2p))},
 \qquad B_p(0):=p-1.
\end{equation}
Thus $B_2=1$ and $B_{p'}=B_p^{-1}$.
Throughout, we set $\beta=\log \|A\|$.

\begin{lemma}\label{lem:monotone}
For $1<p\le r<\infty$, $B_r(h)/B_p(h)$ is nondecreasing for $h\ge0$. 
For $r\ge2$, $B_r$ itself is nondecreasing.
\end{lemma}
\begin{proof}
For $h>0$, set $F_h(u)=u\coth(uh)$, with its continuous value at $u=0$ and $\displaystyle F_h(0)=\frac{1}{h}$. 
Its derivative with respect to $u$ is
\[
 \coth(uh)-uh\,\operatorname{csch}^2(uh)
 =\frac{\sinh(uh)\cosh(uh)-uh}{\sinh^2(uh)}\ge0.
\]
We have that 
\begin{align*}
 2\frac{d}{d h}\log\frac{B_r(h)}{B_p(h)}
 &=F_{h/2}(1-1/r)-F_{h/2}(1-1/p)\\
 &\quad+F_{h/2}(1/p)-F_{h/2}(1/r)\ge0.
\end{align*}
Take $p=2$ for the last assertion and use continuity at zero.
\end{proof}


\begin{lemma}
\label{qubit:J-convexity}
Suppose that \(0<a<1\).
Let
\[
 f(u)=\frac{\sinh((1-a)u)}{\sinh(au)},\qquad
 J_{f}(u)=(\sinh u)\,f(u)-(\cosh u)\,f'(u).
\]
The function \(J_{f}\) extends smoothly and oddly through \(u=0\), vanishes there, and is convex on \(u>0\).
\end{lemma}

\begin{proof}
Note that the function \(f\) is even and analytic at zero, with \(f(0)=(1-a)/a\).  
Hence \(J_{f}\) is odd and vanishes at zero.
Direct differentiation gives
\begin{equation}
\label{qubit:J-second-derivative}
 J_{f}''(u)=\frac{a g_a(2u)}{2\sinh^4(au)},
\end{equation}
where
\begin{align}
\label{qubit:L-polynomial}
 g_a(v)={}&(1-a)^2\sinh((1+a)v)
          +(1+a)^2\sinh((1-a)v)\\
          &-2(1-2a^2)\sinh v-2a\sinh(av).
\nonumber
\end{align}

We shall show that \(g_a(v)\geq0\) for \(v\geq0\) via the Taylor expansion.
Its coefficient of \(v\) is zero.  For \(n\geq1\), the coefficient of
\(v^{2n+1}/(2n+1)!\) equals
\begin{align*}
 C_n={}&(1-a)^2(1+a)^{2n+1}
       +(1+a)^2(1-a)^{2n+1}-2(1-2a^2)-2a^{2n+2}.
\end{align*}
Let \(x=a^2\) and
\[
 S_n(x)=\sum_{j=0}^{n-1}\binom{2n-1}{2j}x^j.
\]
Then we see that 
\begin{align*}
 C_n
 &=(1-a^2)^2\big((1+a)^{2n-1}+(1-a)^{2n-1}\big)
        -2+4a^2-2a^{2n+2}\\
 &=2(1-x)^2S_n(x)-2+4x-2x^{n+1}\\
 &=2\big((1-x)^2(S_n(x)-1)+x^2(1-x^{n-1})\big)\geq0.
\end{align*}
Indeed \(0<x<1\) and \(S_n(x)\geq1\).  
Thus all Taylor coefficients are nonnegative, and \eqref{qubit:J-second-derivative} proves convexity.  
In particular, \(C_1=0\); the first nonzero possible term of \(g_a\) has degree five, consistent with the smooth behavior at zero in \eqref{qubit:J-second-derivative}.
\end{proof}

\begin{lemma}\label{hc:positive-extension}
Let $\Psi:(\cM, \varphi) \to (\cN, \psi)$ be a completely positive map, and
$1\leq p,r<\infty$, where $\varphi$, $\psi$ are faithful normal states on the matrix algebras $\cM, \cN$ respectively.  
Suppose for every $X\geq0$,
\[
 \|\Psi(X)\|_{r, \psi}\leq C\|X\|_{p, \varphi}.
\]
Then the same inequality holds for every $X\in \cM$.
\end{lemma}

\begin{proof}
Define the completely positive map $\widetilde{\Psi}$ between unweighted Schatten spaces
\[
 \widetilde{\Psi}(X)=D_\psi^{1/(2r)}\Psi(D_\varphi^{-1/(2p)}X D_\varphi^{-1/(2p)})D_\psi^{1/(2r)}.
\]
The hypothesis gives $\|\widetilde{\Psi}(X)\|_r\le C \| X\|_p$ for $X\ge0$. 
By the polar decomposition, for any $X\in \cM$, we have that
\[
 \begin{pmatrix}|X^*|&X\\ X^*&|X|\end{pmatrix}\geq0.
\]
Applying the matrix amplification of $\widetilde \Psi$ gives
\[
 \begin{pmatrix}\widetilde{\Psi}(|X^*|) &\widetilde{\Psi}(X)\\ \widetilde{\Psi} (X)^*&\widetilde{\Psi}(|X|)\end{pmatrix}\geq0.
\]
Consequently $\widetilde{\Psi}(X)=\widetilde{\Psi}(|X^*|)^{1/2}K\widetilde{\Psi}(|X|)^{1/2}$ for a contraction $K$;
the factorization remains valid on supports when either diagonal block
is singular.  
The H\"older inequality implies that 
\[
 \|\widetilde{\Psi} (X)\|_{r}
 \leq\|\widetilde{\Psi}(|X^*|)\|_{r}^{1/2}\|\widetilde{\Psi}(|X|)\|_{r}^{1/2}
 \leq C\,\||X^*|\|_{p}^{1/2}\||X|\|_{p}^{1/2}
 =C\|X\|_{p}.
\]
Replacing $X$ by $\bj_\varphi^{1/p}(X)$, we see the lemma is true.
\end{proof}

\section{\texorpdfstring{$2$}{2}-Logarithmic Sobolev Inequality and Hypercontractivity}
In this section, we establish the sharp $2$-logarithmic Sobolev inequality and hypercontractivity through the following steps.

First, we prove the sharp $p$-logarithmic Sobolev inequality on qubits.
Second, we prove sharp hypercontractivity on qubits as a corollary.
Third, we transfer hypercontractivity on qubits to hypercontractivity on the baby Fock model.
Fourth, we apply Speicher's central-limit argument to obtain hypercontractivity on $q$-Araki–Woods algebras.
Finally, we obtain the sharp $2$-logarithmic Sobolev inequality on $q$-Araki–Woods algebras.

\subsection{Sharp \texorpdfstring{$p$}{p}-logarithmic Sobolev inequalities on qubits}
Yangjing Dong, Li Gao, Fengning Ou, Penghui Yao, and Haigang Zhou\cite{DGOYZ} show the exact \(2\)-LSI tensorization for primitive reversible qubit semigroups satisfying their positive off-diagonal scaling condition.
Salman Beigi and Hami Mehrabi\cite{BM} proves the sharp \(2\)-LSI tensorization for finite-mode bosonic Ornstein–Uhlenbeck semigroups.
In this section, we obtain a sharp $p$-LSI for a qubit.

Fix \(h>0\), and put
\begin{equation}
\label{qubit:sigma}
 s_+=\frac{e^h}{2\cosh h},\qquad
 s_-=\frac{e^{-h}}{2\cosh h},\qquad
 \sigma=\begin{pmatrix}s_+&0\\0&s_-\end{pmatrix}.
\end{equation}
We identify $\sigma$ with the state \( \Tr(\sigma \cdot )\) on the matrix algebra $M_2(\CC)$.  
The semigroup generator \(\cL_{h}^{(2)}\) is
\begin{equation}
\label{qubit:generator}
\cL_{h}^{(2)}\begin{pmatrix}x_{11}&x_{12}\\x_{21}&x_{22}\end{pmatrix}
 =
 \begin{pmatrix}
 2s_-(x_{11}-x_{22})&x_{12}\\
 x_{21}&2s_+(x_{22}-x_{11})
 \end{pmatrix}.
\end{equation}
Consequently, the semigroup \(\Theta_t=e^{-t \cL_{h}^{(2)}}\) has the explicit form
\begin{equation}
\label{qubit:semigroup}
 \Theta_t(X)=
 \begin{pmatrix}
  \sigma(X)+e^{-2t}(x_{11}-\sigma(X))
       &e^{-t}x_{12}\\
 e^{-t}x_{21}
       & \sigma(X)+e^{-2t}(x_{22}- \sigma(X))
 \end{pmatrix}.
\end{equation}
This is a completely positive, unital, \(\sigma\)-preserving semigroup.

\begin{theorem}[Sharp \(p\)-LSI on Qubits]
\label{qubit:main-theorem}
For every \(1<p<\infty\), we have that
\begin{equation}\label{qubit:alpha}
 \alpha_p(\cL_{h}^{(2)})=\frac{pp'\,\sinh(h/p)\sinh(h/p')}{2h\sinh h}.
\end{equation} 
In particular, $\displaystyle \alpha_2(\cL_{h}^{(2)})=\frac{\tanh(h/2)}{h}$ and $\displaystyle \alpha_2(\cL_{0}^{(2)})=\frac{1}{2}$.
\end{theorem}
\begin{proof}
Suppose that \(\sigma_x,\sigma_y,\sigma_z\) are the Pauli matrices.  
We assume that $\|X\|_{p, \sigma}=1$, $X>0$ and let $\rho=\bj_\sigma^{1/p}(X)^p$.
Then the strictly positive density \(\rho\) can be written as
\begin{equation}
\label{qubit:bloch-rho}
 \rho=\frac{e^{u\,\boldsymbol n\cdot\boldsymbol\sigma}}{2\cosh u}
 =\frac{I+(\tanh u)\,\boldsymbol n\cdot\boldsymbol\sigma}{2},
\end{equation}
where $u\geq0,\quad |\boldsymbol n|=1$ and $\boldsymbol n=(n_x, n_y, n_z)$, $\boldsymbol\sigma=(\sigma_x, \sigma_y, \sigma_z)$.
Put \(z=n_z\).
By a routine computation, we obtain that
\[
 \log\rho=u\,\boldsymbol n\cdot\boldsymbol\sigma
              -\log(2\cosh u),\qquad
 \log\sigma=h\sigma_z-\log(2\cosh h),
\]
and
\begin{equation}\label{qubit:bloch-entropy}
 D(\rho\Vert\sigma)=u\tanh u-\log\cosh u+\log\cosh h-hz\tanh u.
\end{equation}
Let $\displaystyle k=\left|1-\frac2p\right|\in[0,1)$.
We have that
\begin{align}\label{qubit:bloch-energy}
\frac{4}{p p'}\mathcal E_{p,\sigma}(X)
 ={}&1-\frac{\cosh(kh)}{2\cosh h}\left((1-z^2)+(1+z^2)\frac{\cosh(ku)}{\cosh u}\right)
       -(\tanh h)\,(\tanh u)\,z\\
 &+\frac{\sinh(kh)\sinh(ku)}{\cosh h\cosh u}\,z.
\nonumber
\end{align}

Let
\[
 F(u,z)=\frac{4}{pp'}\mathcal E_{p,\sigma}(X)-\frac{2\sinh(h/p)\sinh(h/p')}{h\sinh h} D(\rho\Vert\sigma)=A(u)+B(u)z+C(u)z^2,
\]
where
\begin{align}
\label{qubit:quadratic-A}
 A(u)&=1-\frac{\cosh(kh)}{2\cosh h}\left(1+\frac{\cosh(ku)}{\cosh u}\right)\notag \\
       &-\frac{2\sinh(h/p)\sinh(h/p')}{h\sinh h}\big(u\tanh u-\log\cosh u+\log\cosh h\big),\\
\label{qubit:quadratic-B}
 B(u)&=\left(h\frac{2\sinh(h/p)\sinh(h/p')}{h\sinh h}-\tanh h\right)(\tanh u)
       +\frac{\sinh(kh)\sinh(ku)}{\cosh h\cosh u},\\
\label{qubit:quadratic-C}
 C(u)&=\frac{\cosh(kh)}{2\cosh h}\left(1-\frac{\cosh(ku)}{\cosh u}\right).
\end{align}
For \(u>0\), \(C(u)>0\).
By a direct computation, we have that $ B(h)=-2C(h)$ and $F(h,1)=0$ since $\rho=\sigma$.
Thus the desired inequality is exactly \(F(u,z)\geq0\).
Let
\[
 G(u)=A(u)-\frac{B(u)^2}{4C(u)}.
\]
To show $F(u, z)\geq 0$, it suffices to prove that $G(u)\geq 0$ by regarding $F(u,z)$ as a quadratic polynomial in $z$.
By taking
\begin{equation}\label{qubit:a-b-f}
 f(u)=\frac{\sinh((1-k)u/2)}{\sinh((1+k)u/2)},\qquad
 g(u)=\frac{\sinh((1+k)u/2)}{\sinh((1-k)u/2)}.
\end{equation}

\begin{align*}
 A_2=\frac{4\sinh^4((1+k)h/2)}{2\sinh^2 h},\qquad
 B_2=\frac{4\sinh^4((1-k)h/2)}{2\sinh^2 h},
\end{align*}
we have that
\begin{align}\label{qubit:G-formula}
 G(u)={}&1-\frac{\cosh^2(kh)+\frac{((cosh h)(\cosh(kh))-1)^2}{\sinh^2 h}}{2(\cosh h)(\cosh(kh))}
 -\frac{\cosh(ku)+A_2f(u)+B_2g(u)}{2(\cosh h)(\cosh(kh))\cosh u}\\
 &-\frac{2\sinh(h/p)\sinh(h/p')}{h\sinh h}\big(u\tanh u-\log\cosh u+\log \cosh h\big).
\nonumber
\end{align}
Moreover, $G(h)=0$.
For any smooth function \(f\), recall that $J_f(u)=(\sinh u)\,f(u)-(\cosh u)\,f'(u)$,
and define
\[
 R(u)=\cosh(ku)\sinh u-k\sinh(ku)\cosh u, \quad V(u)=R(u)+A_2J_f(u)+B_2J_g(u).
\]
Differentiating \eqref{qubit:G-formula} yields
\begin{equation}
\label{qubit:G-derivative}
 G'(u)=(\operatorname{sech}^2u)
 \left(\frac{V(u)}{2(\cosh h)(\cosh(kh))}-u\frac{2\sinh(h/p)\sinh(h/p')}{h\sinh h}\right).
\end{equation}
Moreover,  $G'(h)=0$ and
\[
 \frac{V(h)}{2(\cosh h)(\cosh(kh) )h}=\frac{2\sinh(h/p)\sinh(h/p')}{h\sinh h}.
\]
Differentiating $R(u)$, we have that 
\[
 R'(u)=(1-k^2)\cosh(ku)\cosh u,
\]
\[
 R''(u)=(1-k^2)
 \big(k\sinh(ku)\cosh u+\cosh(ku)\sinh u\big)\geq0.
\]
Thus \(R\) is convex and \(R(0)=0\). 
Applying  Lemma \ref{qubit:J-convexity} with \((1\pm k)/2\) , we see that \(J_f,J_g\) are convex.  
Since \(A_2,B_2\geq0\), the function $V(u)$ is convex on \(u>0\) and satisfies \(V(0)=0\). 
Note that \(\displaystyle uV'(u)-V(u)=\int_0^u sV''(s)\,ds\geq0\).
We see that 
\begin{equation}\label{qubit:V-over-u}
 u\longmapsto \frac{V(u)}u
 \quad\text{is nondecreasing on }(0,\infty).
\end{equation}
Hence
\[
 G'(u)\leq0\quad(0<u\leq h),\qquad
 G'(u)\geq0\quad(u\geq h).
\]
Hence \(G(u)\geq G(h)=0\) for all \(u>0\).
By noting \(\displaystyle f(0)=\frac{1-k}{1+k}\) and \(\displaystyle g(0)=\frac{1+k}{1-k}\), we have that \eqref{qubit:G-formula} extends continuously to zero.  
The same conclusion therefore holds at \(u=0\).
This shows \(F(u,z)\geq0\) for every \(z\in[-1,1]\), and even for every real \(z\) when \(u>0\).
This proves the inequality in Theorem \ref{qubit:main-theorem}.  

Approximation by \(X+\varepsilon I\) extends the result to semi-definite \(X\): all matrix powers involved have positive exponents and \(t\log t\) is continuous at zero.  
Finally, passing to the limit \(h\downarrow0\) proves the stated tracial value.

In the following, we show the sharpness of the inequality.
Let
\[
 Q=\begin{pmatrix}0&1\\1&0\end{pmatrix},\qquad
 X_\varepsilon=I+\varepsilon Q,\qquad |\varepsilon|<1.
\]
The local norm expansion is
\begin{equation}
\label{qubit:local-norm}
 \|X_\varepsilon\|_{p,\sigma}
 =1+\frac{B_p(2h)}{2\cosh h}\varepsilon^2+O(\varepsilon^4).
\end{equation}
Here is a direct verification.  
The weighted matrix is
\[
 \bj_\sigma^{1/p}(X_\varepsilon)
 =
 \begin{pmatrix}
 s_+^{1/p}&\varepsilon(s_+s_-)^{1/(2p)}\\
 \varepsilon(s_+s_-)^{1/(2p)}&s_-^{1/p}
 \end{pmatrix}.
\]
The second-order expansion of the trace of its \(p\)-th power gives
\[
 \Tr\!\left(\bj_\sigma^{1/p}(X_\varepsilon)^p\right)
 =1+p(s_+s_-)^{1/p}
   \frac{s_+^{1-1/p}-s_-^{1-1/p}}
        {s_+^{1/p}-s_-^{1/p}}\varepsilon^2
   +O(\varepsilon^4).
\]
Taking the \(p\)-th root yields \eqref{qubit:local-norm}, because
\[
 (s_+s_-)^{1/p}
 \frac{s_+^{1-1/p}-s_-^{1-1/p}}
      {s_+^{1/p}-s_-^{1/p}}
 =\frac1{2\cosh h}
   \frac{\sinh(h/p')}{\sinh(h/p)}.
\]
The continuous value at \(h=0\) is \(\displaystyle \lim_{h\to 0}\frac{B_p(2h)}{2\cosh h}=(p-1)/2\).

Near any $(p_0, 0)$, $p_0>1$ the weighted matrix is uniformly positive, and its $p$-norm is jointly real-analytic in $(p, \varepsilon)$. 
Since the norm is even in $\varepsilon$, the expansion \eqref{qubit:local-norm} is differentiated with respect to $p$, locally uniformly in $p$.
The norm derivatives \eqref{qubit:p-norm-derivative} and
\eqref{qubit:time-norm-derivative}, or direct differentiation of the
same local expansion, give
\begin{align}
\label{qubit:local-entropy}
 \Ent_{p,\sigma}(X_\varepsilon)
 &=p^2\frac{d}{dp}\frac{B_p(2h)}{2\cosh h}\,\varepsilon^2+O(\varepsilon^4)
 =\frac{h\tanh h}{2\sinh^2(h/p)}\varepsilon^2
   +O(\varepsilon^4),\\
\label{qubit:local-energy}
 \mathcal E_{p,\sigma}(X_\varepsilon)
 &=\frac{pp'}{2}\frac{B_p(2h)}{2\cosh h}\,\varepsilon^2
   +O(\varepsilon^4).
\end{align}
For \eqref{qubit:local-energy}, use \(\Theta_t(X_\varepsilon )=I+e^{-t}\varepsilon Q\).  We have also used the identity
\begin{equation}
\label{qubit:log-B-derivative}
 \frac{d}{dp}\log B_p(2h)
 =\frac{h}{p^2}
    \left(\coth(h/p)+\coth(h/p')\right)
 =\frac{h\sinh h}
       {p^2\sinh(h/p)\sinh(h/p')}.
\end{equation}
It follows that
\[
 \lim_{\varepsilon\to0}
 \frac{\mathcal E_{p,\sigma}(X_\varepsilon)}
      {\Ent_{p,\sigma}(X_\varepsilon)}
 =
 \frac{pp'\,\sinh(h/p)\sinh(h/p')}
      {2h\sinh h}
 =\alpha_p(\cL_h^{(2)}).
\]
No larger constant can satisfy \eqref{qubit:lsi-convention}.  
This completes the sharpness assertion of Theorem \ref{qubit:main-theorem}.
\end{proof}

\subsection{Sharp hypercontractivity on Qubits}\label{qubit:hypercontractivity}

In this section, we obtain the sharp hypercontractivity on qubits by using the $p$-logarithmic Sobolev inequality.
\begin{corollary}\label{qubit:HC}
Let \(1<p\leq r<\infty\).  
The inequality
\[
 \|\Theta_t(X)\|_{r,\sigma}\leq\|X\|_{p,\sigma}
\]
holds for every \(X\in M_2(\mathbb C)\), if and only if
\begin{equation}
\label{qubit:HC-time}
 t\geq t_*(p,r;h)
 :=\frac12\log\frac{B_r(2h)}{B_p(2h)}
 =\frac12\log
 \frac{\sinh(h/r')\sinh(h/p)}
      {\sinh(h/r)\sinh(h/p')}.
\end{equation}
At \(h=0\), this becomes
\[
 t_*(p,r;0)=\frac12\log\frac{r-1}{p-1}.
\]
\end{corollary}

\begin{proof}
We first prove sufficiency for \(X>0\).  Let \(s(t)\) increase from
\(p\) according to
\[
 \dot s(t)=4(s(t)-1)\alpha_{s(t)}(\cL_h^{(2)}).
\]
For \(Z(t)=\|\Theta_t(X)\|_{s(t),\sigma}^{s(t)}\), the two differentiation identities above give
\[
 \frac{d}{dt}\log\|\Theta_t(X)\|_{s(t),\sigma}
 =
 \frac{\dot s(t)\Ent_{s(t),\sigma}(\Theta_t(X))
       -4(s(t)-1)\mathcal E_{s(t),\sigma}(\Theta_t(X))}
      {s(t)^2 Z(t)}
 \leq0.
\]
Theorem \ref{qubit:main-theorem} proves the final inequality.  
The time required to move from \(p\) to \(r\) is
\[
 \int_p^r\frac{ds}{4(s-1)\alpha_s(\cL_h^{(2)})}.
\]
Equations \eqref{qubit:alpha} and \eqref{qubit:log-B-derivative} imply
\[
 \frac1{4(s-1)\alpha_s(\cL_h^{(2)})}
 =\frac12\frac{d}{ds}\log B_s(2h),
\]
so the integral is exactly \eqref{qubit:HC-time}.  
The fixed-exponent contractivity of \(\Theta_t\) proves sufficiency for all larger times.
Alternatively, that contractivity follows from \eqref{qubit:time-norm-derivative}, since the proved LSI and positivity of relative entropy imply \(\mathcal E_{s,\sigma}\geq0\).
Approximation extends this argument to \(X\geq0\).
By Lemma \ref{hc:positive-extension}, we see that the hypercontractivity holds for all $X\in M_2(\CC)$.

For necessity, apply the asserted contraction to \(X_\varepsilon=I+\varepsilon Q\).  
Since \(\Theta_t(X_\varepsilon)=I+e^{-t}\varepsilon Q\),
Equation \eqref{qubit:local-norm} gives
\[
 1+e^{-2t}\frac{B_r(2h)}{2\cosh h}\varepsilon^2+O(\varepsilon^4)
 \leq 1+\frac{B_p(2h)}{2\cosh h}\varepsilon^2+O(\varepsilon^4).
\]
Letting \(\varepsilon\to0\), we obtain
\[
 e^{-2t}\leq \frac{B_p(2h)}{B_r(2h)},
\]
which is equivalent to \eqref{qubit:HC-time}.
This completes the proof of the corollary.
\end{proof}

\subsection{Hypercontractivity on the baby-Fock model}
Firstly, we recall the twisted baby Fock model \cite{Nou06} using the notations in \cite[Section 3]{LR}.
Fix $n$ modes and signs $\varepsilon_{ij}=\varepsilon_{ji}\in\{-1,1\}$ for $i\ne j$.  
The signed-spin generators satisfy
\begin{align}
\gamma_j^2&=0,&\gamma_j\gamma_j^*+\gamma_j^*\gamma_j&=(\mu_j^2+\mu_j^{-2})\Id, \notag \\
 \gamma_i\gamma_j&=\varepsilon_{ij}\gamma_j\gamma_i,&
 \gamma_i^*\gamma_j&=\varepsilon_{ij}\gamma_j\gamma_i^*,\qquad i\ne j.
 \label{spin:relations}
\end{align}
Let $\Gamma_{\langle 1, \ldots ,n \rangle}$ be the algebra generated by $\gamma_1, \ldots, \gamma_n$.
In an occupation-number realization its reference state is the product of the thermal states
\begin{equation}
 \sigma_j=\begin{pmatrix}a_j&0\\0&b_j\end{pmatrix},\qquad
 a_j=\frac{1}{1+\mu_j^4},\quad
 b_j=\frac{\mu_j^4}{1+\mu_j^4}.
 \label{spin:state}
\end{equation}
Its Ornstein--Uhlenbeck semigroup $\Phi_t^{(n)}$ multiplies a Wick word of
degree $k$ by $e^{-kt}$, where the basis for the Wick words is the ordered products $w_1\cdots w_n$,  $w_i\in \{1, \gamma_i, \gamma_i^*, y_i=\gamma_i^*\gamma_i-\mu_i^{-2}\Id: 1\leq i\leq n\}$ and the four local choices have degrees $0, 1,1,2$ respectively.
The degree of a product is the sum of its local degrees.
Denote by $\cL_n^{\varepsilon}$ the generator of the semigroup.
By \cite[Proposition 3.3]{LR}, there exists a $*$-isomorphism between $ \Gamma_{\langle 1, \ldots, n\rangle}$ and $M_2(\CC)\otimes \Gamma_{\langle 1, \ldots, n-1\rangle}$ and a self-adjoint unitary 
\begin{equation}\label{eq:signunitary}
 u_n=\prod_{\substack{i<n\\\varepsilon_{in}=-1}}\left(1-2(\mu_i^2+\mu_i^{-2})^{-1}\gamma_i^*\gamma_i\right) \in \Gamma_{\langle 1, \ldots, n-1\rangle}
\end{equation}
such that $\gamma_n$ is identified as $(\mu_n^2+\mu_n^{-2})^{1/2}e_{21}\otimes u_n$.
In particular, $u_n$ belongs to the centralizer of the previous product state $\tau_{n-1}^\varepsilon$.
Conjugation by $u_n$ changes the signs of a specified set of generators;
therefore it commutes with the semigroup $\Phi_t^{(n-1)}$:
\begin{equation}
 \Phi_t^{(n-1)}(u_nxu_n)=u_n\Phi_t^{(n-1)}(x)u_n.
 \label{spin:parity}
\end{equation}
The state on the baby Fock model is $\sigma_n\otimes\tau_{n-1}^{\varepsilon}$.

Let $W=\operatorname{diag}(1,u_n)$ and define the completely positive unital
map
\begin{equation}
 S_t=\operatorname{Ad}_W\circ(\Id_{M_2(\CC)}\otimes \Phi_t^{(n-1)})
             \circ\operatorname{Ad}_W.
 \label{spin:Sdef}
\end{equation}
We have that $S_t$ preserves $\sigma_n\otimes \tau_{n-1}^{\varepsilon}$.  
For $X=\begin{pmatrix}X_0 &B\\C&D\end{pmatrix}\in \Gamma_{\langle 1, \ldots, n\rangle }$,
Equation \eqref{spin:parity} implies the concrete formula
\begin{equation}
 S_t(X)=
 \begin{pmatrix}
  \Phi_t^{(n-1)}(X_0)&\Phi_t^{(n-1)}(Bu_n)u_n\\
  u_n \Phi_t^{(n-1)}(u_n C)&\Phi_t^{(n-1)}(D)
 \end{pmatrix}.
 \label{spin:Sblocks}
\end{equation}

The last-mode semigroup is $T_t=\Theta_{t,\sigma_n}\otimes\Id_{\Gamma_{\langle 1, \ldots, n-1\rangle}}$, given by
\begin{align}
 & \Theta_{t,\sigma_n}
 \begin{pmatrix}x_{11}&x_{12}\\x_{21}&x_{22}\end{pmatrix}\notag \\
 =& \begin{pmatrix}
 (a_n+e^{-2t} b_n)x_{11}+b_n(1-e^{-2t})x_{22} &e^{-t} x_{12}\\
 e^{-t} x_{21} & a_n(1-e^{-2t})x_{11}+(b_n+e^{-2t} a_n)x_{22}
 \end{pmatrix}.
 \label{spin:T}
\end{align}
This is unitarily conjugate to a qubit semigroup in \eqref{qubit:semigroup} with parameter $h_n=2\log \mu_n$, i.e.
\begin{align*}
 \sigma_n= U \sigma_{h_n} U, \quad \Theta_{t, \sigma_n}= \operatorname{Ad}_U \Theta_{t, \sigma_{h_n}} \operatorname{Ad}_U, 
\end{align*}
where $U=\begin{pmatrix}
    0 & 1 \\ 1& 0
\end{pmatrix}$.
Formula \eqref{spin:Sblocks} shows directly that $S_tT_s=T_sS_t$ for all $s,t\geq0$: 
the diagonal blocks are acted on by the same $\Phi_t^{(n-1)}$, and the off-diagonal blocks are merely multiplied by $e^{-s}$. 
Moreover,
\begin{equation}
 \Phi_t^{(n)}=S_tT_t.
 \label{spin:factor}
\end{equation}

To prove hypercontractivity for the baby-Fock model, we need the following results.
For $\displaystyle X=\begin{pmatrix}{X_0}&B\\B^*&C\end{pmatrix}\ge0$ and a faithful matrix state $\psi$, define
\[
 Q_s(X)=\begin{pmatrix}
 \norm {X_0}_{s,\psi}&\norm B_{s,\psi}\\
 \norm B_{s,\psi}&\norm C_{s,\psi}
 \end{pmatrix}.
\]
Conjugation of $X$ by $\operatorname{diag}(D_\psi^{1/(2s)},D_\psi^{1/(2s)})$ preserves positivity. 
Positive block factorization and Schatten H\"older imply
\[
 \norm B_{s,\psi}^2\le\norm {X_0}_{s,\psi}\norm C_{s,\psi},
\]
so $Q_s(X)\ge0$. 
Applying King's inequality~\cite[Theorem 1]{King} to the further conjugation by $\sigma^{1/(2s)}$, where $\sigma=\operatorname{diag}(a,b)$, gives
\begin{equation}\label{eq:king}
 \norm X_{s,\sigma\otimes\psi}\le\norm{Q_s(X)}_{s,\sigma},\qquad 2\le s<\infty.
\end{equation}
For $1\le s\le2$ the direction is reversed. At $s=2$ it is an equality, since
\begin{equation}\label{eq:compression2}
 \norm X_{2,\sigma\otimes\psi}^2
 =a\norm {X_0}_{2,\psi}^2+b\norm C_{2,\psi}^2
    +2\sqrt{ab}\norm B_{2,\psi}^2
 =\norm{Q_2(X)}_{2,\sigma}^2.
\end{equation}

We also need a scalar fact. 
If $U,V$ are positive real $2\times2$ matrices with nonnegative entries and $(U)_{jk}\le (V)_{jk}$, then $\norm U_s\le\norm V_s$ for $s\ge1$, where $(U)_{jk}$ is the $(j,k)$-entry of $U$. 
This observation is not true for matrices of large size. (See \cite{GKR} for more details.)
To see it, let $U_\theta=(1-\theta)U+\theta V$. 
For $s>1$,
\[
 \frac d{d\theta}\Tr ( U_\theta^s)=s\Tr\bigl(U_\theta^{s-1}(V-U)\bigr)\ge0.
\]
Indeed the diagonal entries of $U_\theta^{s-1}$ are nonnegative, and its off-diagonal entry is
\[
 (U_\theta)_{jk}\,
 \frac{\lambda_+^{s-1}-\lambda_-^{s-1}}{\lambda_+-\lambda_-}\ge0, \quad j\neq k, 
\]
where $\lambda_{\pm}$ are eigenvalues of $U_\theta$ and $\lambda_+\geq \lambda_-$.
At \(\lambda_+=\lambda_-=\lambda>0\), the quotient is interpreted by continuity as
\[
(s-1)\lambda^{s-2}.
\]
The case \(s=1\) follows directly from the trace.
Apply the argument to $U+\epsilon I$ and $V+\epsilon I$, then let $\epsilon \downarrow 0$. 
Diagonal weight conjugation preserves the hypotheses, proving the weighted version as well.

\begin{proposition}\label{spin:finiteHC}
Suppose that $2\leq r<\infty$ and there exists $\Lambda>0$ such that $1\leq\mu_j^4\leq\Lambda<\infty$ for all $j$.
Let $\displaystyle h=\frac12\log\Lambda$.  
If
\begin{equation}
 e^{-2t}\leq B_r(2h)^{-1},
 \label{spin:time}
\end{equation}
then, for every choice of the signs and every number of modes,
\begin{equation}
 \|\Phi_t^{(n)}(X)\|_{r,\tau_n^{\varepsilon}}
 \leq\|X\|_{2,\tau_n^{\varepsilon}},
 \qquad X\in  \Gamma_{\langle 1, \ldots, n\rangle}.
 \label{spin:finitecontraction}
\end{equation}
\end{proposition}

\begin{proof}
The one-mode case follows from Corollary \ref{qubit:HC}, since $B_2=1$ and $B_r(4\log\mu_1)\le B_r(\Lambda)$ by Lemma~\ref{lem:monotone}.

We shall prove by induction.
Assume the result for $n-1$ modes and initially take
$X=\begin{pmatrix}X_0&B\\B^*&D\end{pmatrix}\geq0$.  
Write
\begin{equation}
Q_X=
 \begin{pmatrix}
  \|X_0\|_{2,\tau_{n-1}^{\varepsilon}}&\|B\|_{2,\tau_{n-1}^{\varepsilon}}\\
  \|B\|_{2,\tau_{n-1}^{\varepsilon}}&\|D\|_{2,\tau_{n-1}^{\varepsilon}}
 \end{pmatrix}\geq0.
 \label{spin:CX}
\end{equation}
Denote the four nonnegative diagonal coefficients in \eqref{spin:T} by $q_{00},q_{01},q_{10},q_{11}$. 
By induction and the triangle inequality, the first diagonal block of $S_tT_t(X)$ satisfies
\begin{equation}
 \left\|q_{00}\Phi_t^{(n-1)}(X_0)+q_{01}\Phi_t^{(n-1)}(D)\right\|_{r,\tau_{n-1}^{\varepsilon}}
 \leq q_{00}\|X_0\|_{2,\tau_{n-1}^{\varepsilon}}+q_{01}\|D\|_{2,\tau_{n-1}^{\varepsilon}},
 \label{spin:diagonalbound}
\end{equation}
and likewise for the other diagonal block.  
Since $u_n$ centralizes $\tau_{n-1}^{\varepsilon}$, left and right multiplication by $u_n$ are isometries.  
Consequently the off-diagonal block satisfies
\begin{align}
& \|e^{-t} \Phi_t^{(n-1)}(Bu_n)u_n\|_{r,\tau_{n-1}^{\varepsilon}}
 =e^{-t}\|\Phi_t^{(n-1)}(Bu_n)\|_{r,\tau_{n-1}^{\varepsilon}}\notag \\
 \leq & e^{-t}\|Bu_n\|_{2,\tau_{n-1}^{\varepsilon}}
 =e^{-t}\|B\|_{2,\tau_{n-1}^{\varepsilon}}.
 \label{spin:offdiagonalbound}
\end{align}
Thus the positive compression matrix is bounded entry by entry by
$\Theta_{t,\sigma_n}(Q_X)$. 
Hence by the argument before the proposition, we have that
\begin{align}
 \|\Phi_t^{(n)}(X)\|_{r,\sigma_n\otimes\tau_{n-1}^{\varepsilon}}
 &\leq\|\Theta_{t,\sigma_n}(Q_X)\|_{r,\sigma_n}\notag\\
 &\leq\|Q_X\|_{2,\sigma_n}
 =\|X\|_{2,\sigma_n\otimes\tau_{n-1}^{\varepsilon}}.
 \label{spin:induction}
\end{align}
The last equality is true, because the square of both can be written as 
\begin{equation}
 a_n\|X_0\|_{2,\tau_{n-1}^{\varepsilon}}^2+b_n\|D\|_{2,\tau_{n-1}^{\varepsilon}}^2
       +2\sqrt{a_nb_n}\,\|B\|_{2,\tau_{n-1}^{\varepsilon}}^2.
 \label{spin:HSequality}
\end{equation}
Applying Lemma~\ref{hc:positive-extension}, we see the desired inequality is true for all $X\in \Gamma_{\langle 1, \ldots, n\rangle}$ and this completes the proof of the proposition.
\end{proof}

\begin{remark}
Lee–Ricard formulate their estimates using right embeddings. Interpolation between the right and left embedding inequalities, followed by KMS duality, gives a sufficient bound for the symmetric norms used here. 
Proposition \ref{spin:finiteHC} gives the sharp corresponding \(2\to r\) bound.
\end{remark}

\subsection{Hypercontractivity on \texorpdfstring{$q$}{q}-Araki--Woods algebras}
\label{spin:CLTsection}

In this section, we shall apply Speicher's central-limit procedure.
Following the lines in \cite[Section 4]{LR}, we obtain generalized baby gaussians $\gamma_{i, j}$ associated with the parameter $\lambda_i=\mu_i^4$ for $1\leq i \leq n$ and $1\leq j \leq m$.
Let $\Gamma_{n,m}$ be the von Neumann algebra generated by $\gamma_{i,j}$ and let $\tau_{m}^\varepsilon$ be the product state.
Denote by $D_m$ the density of $\tau_{m}^\varepsilon$.
Let $\displaystyle s_{i,m}=\frac{1}{\sqrt{m}}\sum_{j=1}^m \gamma_{i, j}$.

Let $\mathcal R$ be the Groh–Raynaud ultraproduct of the baby-Fock models $(\Gamma_{n,m},\tau_{m}^{\varepsilon})$, let $\omega=(\tau_{m}^{\varepsilon})_\mathcal U$,
and let $e=s(\omega)$ be its support. 
Nou's embedding identifies $(\Gamma_q,\varphi)$ with a modularly invariant, expected subalgebra of $(e\mathcal R e,\omega)$; 
denote this embedding by $\iota$.
Nou's Lemma 4.1 ensures that $e$ commutes with the algebra generated by the bounded cutoff representatives of the Gaussian generators.
In particular compression to $e$ is multiplicative on that algebra.

For a fixed $t$, the maps $\Phi_t^{(m)}$ induce a normal unital
completely positive map $\mathcal Q_t$ on $\mathcal R$. 
One obtains it by taking the adjoint of the ultraproduct of the predual maps;
on uniformly bounded representatives it is given by
\[
 \mathcal Q_t([Z_m]_{\mathcal U})
       =[\Phi_t^{(m)}(Z_m)]_{\mathcal U}.
\]
It preserves $\omega$ and is GNS-symmetric with respect to $\omega$, even though $\omega$ need not be faithful on $\mathcal R$.

For any $z\in \mathcal R$ with $ze=0$, we have that 
\[
 \omega(\mathcal Q_t(z)^*\mathcal Q_t(z))
 \leq\omega(\mathcal Q_t(z^*z))=\omega(z^*z)=0,
\]
so $\mathcal Q_t(z)e=0$.  If $b=(1-e)be$, then $b^*e=0$;
applying the preceding observation to $b^*$ gives
$e\mathcal Q_t(b)=0$.  
For $a\in\{e\}'$, take $b=(1-e)\mathcal Q_t(a)e$.  
The GNS symmetry gives
\[
 \omega(b^*b)
  =\omega(\mathcal Q_t(a)^*b)
  =\omega(a^*\mathcal Q_t(b))=0,
\]
where the last equality uses $a e=e a$ and
$e\mathcal Q_t(b)=0$.  Faithfulness on $e\mathcal R e$ gives
$(1-e)\mathcal Q_t(a)e=0$.  Applying the same argument to $a^*$
gives $e\mathcal Q_t(a)(1-e)=0$.  Thus
$\mathcal Q_t(\{e\}')\subseteq\{e\}'$ in $\mathcal R$.

The embedding $\iota$ intertwines the semigroups $\Phi_t$ and $\mathcal{Q}_t$:
\[
 \overline{\mathcal Q}_t\circ\iota
 =\iota\circ\Phi_t,
 \qquad
 \overline{\mathcal Q}_t(z)=e\mathcal Q_t(z)e.
\]
This follows from the approximation results of
Lee--Ricard~\cite[Proposition 4.2, proof of Lemma 4.3,
and Lemma 4.5]{LR}, together with the GNS symmetry of
the finite-model and limiting semigroups.

Raynaud's isometries and their product compatibility give, for
every $s>0$ and every uniformly bounded self-adjoint lift
$a=[Z_m]_{\mathcal U}\in\{e\}'$ with $eae=\iota(Z)$,
\begin{align}\label{eq:raynaudnorm}
 \lim_{m,\mathcal U}\|Z_m\|_{s,\tau_m^{\varepsilon}}
     =\|Z\|_{s,\varphi}.
\end{align}
Indeed, $[D_m^{1/(2s)}]_{\mathcal U}$ corresponds to $D_\omega^{1/(2s)}$ in $L^{2s}(\mathcal R)$, and the symmetric weighted products correspond to $D_\omega^{1/(2s)}\iota(Z)D_\omega^{1/(2s)}$.
The expected, state-preserving embedding identifies this with the
canonical isometric image of the symmetric density of $Z$.
This uses Raynaud's Theorems 3.6 and 5.1 in \cite{Raynaud} for every $s>0$.

\begin{lemma}\label{lem:approximation}
For any $*$-polynomial $Q$ in $n$ non-commuting variables and $2\leq r <\infty$ and every $t\geq 0$, we have that
\begin{align*}
    \lim_{m, \mathcal U} \left\| D_m^{\frac{1}{2r}} \Phi_t^{(m)}(Q(s_{1, m}, \ldots, s_{n,m})) D_m^{\frac{1}{2r}} \right\|_r
    =\left\| D_\varphi^{\frac{1}{2r}} \Phi_t(Q(g_{q, 1}, \ldots, g_{q,n})) D_\varphi^{\frac{1}{2r}} \right\|_{L^r(\cM)}.
\end{align*}
\end{lemma}
\begin{proof}
Let $\lambda_i=\mu_i^4$ and $s_m=(s_{1,m}, \ldots, s_{n,m})$, $g=(g_{q,1}, \ldots, g_{q,n})$.
For a polynomial $Q$ and a fixed $s>0$, let $Q^{[s]}$ be obtained by making the substitutions
\[
 z_i\longmapsto\lambda_i^{1/(2s)}z_i,
 \qquad z_i^*\longmapsto\lambda_i^{-1/(2s)}z_i^*.
\]
This is a fixed polynomial, although the substitution is not a $*$-homomorphism. Modular covariance gives the exact identities
\begin{align}\label{eq:modtranslate}
 D_m^{1/(2s)}Q(s_m)D_m^{1/(2s)}
 &=\sigma_{-i/(2s)}^{\tau_{m}^{\varepsilon}}(Q(s_m))D_m^{1/s}
 =Q^{[s]}(s_m)D_m^{1/s},\notag\\
 D_\varphi^{1/(2s)}Q(g)D_\varphi^{1/(2s)}
 &=Q^{[s]}(g)D_\varphi^{1/s}.
\end{align}
Apply the right-embedding convergence to $Q^{[s]}$. 
Since the number maps commute with the modular automorphisms, the identical argument applies after those maps. Thus for $s\ge1$ we have
\begin{align}\label{eq:normCLT}
 \lim_{m,\mathcal U}\norm{Q(s_m)}_{s,\tau_{m}^{\varepsilon}}
 &=\norm{Q(g)}_{s,\varphi},\notag\\
 \lim_{m,\mathcal U}\norm{\Phi_t^{(m)}(Q(s_m))}_{s,\tau_{m}^{\varepsilon}}
 &=\norm{\Phi_t(Q(g))}_{s,\varphi}.
\end{align}
The first line for every finite $s\ge1$ follows from Lee--Ricard's Lemma 4.3 and Remark 4.4. 
To spell out the additional interpolation step, write $\Phi_t(Q(g))=R_t(g)$ using the finite Wick decomposition of the fixed polynomial $Q$. 
By the preceding semigroup intertwining and the cutoff comparison in the proof of \cite[Lemma 4.3]{LR}, applied to $Q$ and $R_t$, we obtain
\[
 Z_m=\Phi_t^{(m)}(Q(s_m))-R_t(s_m),\qquad
 \lim_{m,\mathcal U}\norm{Z_m}_{2,\tau_{m}^{\varepsilon}}=0.
\]
For any prescribed $r>2$, choose $u>r$. 
Fixed-exponent $L^u$ contractivity of $\Phi_t^{(m)}$ and the first line of~\eqref{eq:normCLT} give a uniform bound on $\norm{Z_m}_{u,\tau_{m}^{\varepsilon}}$. 
If $1/r=(1-\theta)/2+\theta/u$, symmetric Kosaki interpolation yields
\[
 \norm{Z_m}_{r,\tau_{m}^{\varepsilon}}
 \le\norm{Z_m}_{2,\tau_{m}^{\varepsilon}}^{1-\theta}
       \norm{Z_m}_{u,\tau_{m}^{\varepsilon}}^{\theta}\longrightarrow0.
\]
For $1\le r\le2$, monotonicity of the symmetric state norms gives the conclusion directly from $L^2$.
Thus the second line holds at every finite exponent needed here. 
Restricting the cited semigroup approximation to its initially stated range would not suffice.

\end{proof}

\begin{theorem}\label{thm:cross2r}
Suppose that $\{\Phi_t\}_{t\geq 0}$ is a $q$-Araki-Woods Ornstein-Uhlenbeck semigroup with $\|A\|<\infty$.
Then for $2\leq r<\infty$, 
    \begin{equation}
 \|\Phi_t(x)\|_{r,\varphi}\leq\|x\|_{2,\varphi},\qquad
 e^{-2t}\leq \frac{1}{B_r(\beta)}.
 \label{spin:AWsufficiency}
\end{equation}
\end{theorem}
\begin{proof}
For finitely many modular eigenmodes, apply Proposition~\ref{spin:finiteHC} to every polynomial representative and pass to the ultralimit using~\eqref{eq:normCLT}, with $s=2$ and $s=r$. 
All constants depend only on the common upper bound for $\mu_i^4$.

For infinitely many almost-periodic modes, use increasing finite-mode algebras and their vacuum-preserving conditional expectations. 
By the fact that polynomials in finitely many modes are dense in the $L^2(\cM)$ space and the Kaplansky density theorem, we see that the hypercontractivity holds.

For general bounded $\widetilde{A}=\log A$, choose real odd finite-valued step functions $f_m$ such that $|f_m|\le\beta$ and $f_m\to\id$ uniformly on $[-\beta,\beta]$. 
Then $\widetilde{A}_m=f_m(\widetilde{A})$ has sign symmetry; $A_m$ has reciprocal spectral symmetry, $A_m=e^{\widetilde{A}_m}$ is almost-periodic, and $\norm{A_m}\le e^\beta$. 
The covariance forms converge, and the real field representatives and their imaginary modular translates have uniform operator bounds. 
These are the hypotheses in Nou's discretization and embedding argument~\cite[Sections 6.1--6.2]{Nou06}.
Its state-preserving embedding is compatible with number second quantization:
the Wick expansion coefficients of a fixed polynomial are finite expressions in the converging covariances, so damping each Wick degree before or after passage to the ultraproduct gives representatives whose difference tends to zero in operator norm.
Raynaud's product-compatible $L^s$ embeddings then transfer the uniform estimate to the symmetric norms. 

\end{proof}

Suppose that $\cM$ is a von Neumann algebra with a normal faithful state $\varphi$.
Let $y\in \cM$ be strictly positive and $y^{-1}$ is bounded, $a\in \cM$ with
$\norm a<1$. 
Put
\[
 y_a=y^{1/2}(1+a)y^{1/2},\qquad
 \Lambda_y(a)=\log y_a-\log y,
 \qquad
 \mathbf K_y(a)=\int_0^\infty y_t a y_t\,dt,
\]
where $y_t=y^{1/2}(y+t)^{-1}$. 
The principal logarithm is used when $a$ is not self-adjoint. 
The resolvent identity gives
\[
 (y+t)^{-1}-(y_a+t)^{-1} =y_t a (1+y^{1/2}y_t a)^{-1}y_t.
\]
Note that for $v_t\in \cM$,
\begin{align*}
\norm{\int_0^\infty y_t v_t y_t\,dt}\le\sup_t\norm{v_t}.
\end{align*}
Integrating the resolvent identity and using the Neumann series, we have that
\begin{equation}\label{eq:logremainder0}
 \norm{\Lambda_y(a)}\le\frac {\norm a}{1-\norm a},
 \qquad
 \norm{\Lambda_y(a)-\mathbf K_y(a)}\le\frac{\norm a^2}{1-\norm a}.
\end{equation}

The same logarithm estimates hold when $y$ is an injective positive operator affiliated with a von Neumann algebra $\mathcal M$ and
$a=a^*\in\mathcal M$ with $\|a\|<1$.  
Define $y_a$ by the closed positive form
\[
 \langle y_a\xi,\xi\rangle
 =\langle(1+a)y^{1/2}\xi,y^{1/2}\xi\rangle,
 \qquad \xi\in\operatorname{dom}(y^{1/2}).
\]
In this formula the left side denotes the associated quadratic form.
Put $y_t=y^{1/2}(y+t)^{-1}$ and $\widetilde y_t=y(y+t)^{-1}$.
Factoring the closed form of $y_a+t$ through $(y+t)^{1/2}$ gives
\[
 (y+t)^{-1}-(y_a+t)^{-1}
 =y_ta (1+\widetilde{y}_t a)^{-1}y_t.
\]
Since $0\le \widetilde{y}_t\le1$ and
$\displaystyle \int_0^\infty y_t^2\,dt=1$ strongly, the weak integrals
\[
 \mathbf K_y(a)=\int_0^\infty y_ta y_t\,dt,
 \qquad
 w_a=\int_0^\infty y_ta (1+\widetilde{y}_t a )^{-1}y_t\,dt
\]
are bounded. 
Their truncated integrals also converge strongly: for any uniformly bounded family $v_t$ and any measurable $E\subset(0,\infty)$,
\[
 \left\|\int_E y_t v_t y_t\xi\,dt\right\|
 \le\sup_t\|v_t\|
      \left(\int_E\|y_t\xi\|^2\,dt\right)^{1/2}.
\]
Apply this to the two tails as their endpoints tend to $0$ and $\infty$.
The bounds on the middle factors give
\[
 \|w_a\|\le\frac {\|a\|}{1-\|a\|},\qquad
 \|w_a-\mathbf K_y(a)\|\le\frac{\|a\|^2}{1-\|a\|}.
\]
To identify $w_a$ with the logarithmic difference, use the bounded real functions
\[
 g_{\epsilon,R}(s)=\int_\epsilon^R
          \left(\frac1{1+t}-\frac1{s+t}\right)dt.
\]
Then
\[
 g_{\epsilon,R}(y_a)-g_{\epsilon,R}(y)
 =\int_\epsilon^R y_ta (1+\widetilde{y}_t a)^{-1}y_t\,dt
 \longrightarrow w_a
\]
strongly, with a uniform operator-norm bound. 
Since $g_{\epsilon,R}(s)\to\log s$ for every $s>0$, the spectral theorem
gives strong resolvent convergence of $g_{\epsilon,R}(y)$ and $g_{\epsilon,R}(y_a)$ to $\log y$ and $\log y_a$.
Stability under uniformly bounded, strongly convergent self-adjoint perturbations therefore implies
\[
 \log y_a=\log y+w_a,
 \qquad\operatorname{dom}\log y_a=\operatorname{dom}\log y.
\]
Thus $\Lambda_y(a)=w_a$ is indeed the bounded logarithmic difference.
For $y=D_\varphi$ in the Haagerup crossed product and $a\in \cM$,
the dual action scales both densities by the same scalar.  
Their logarithmic difference is therefore fixed by the dual action and belongs
to $\cM$.

\begin{lemma}\label{lem:hessian}
Let $a=a^*$ be bounded and centered, and suppose its modular spectral support is compact. 
Set $\zeta=\Delta_\varphi^{1/4}a\Omega$. 
For every $0<s<\infty$, as $\varepsilon\downarrow 0$, we have that
\begin{equation}\label{eq:hessian}
 \norm{1+\varepsilon a}_{s,\varphi}^2
 =1+\varepsilon^2\ip{\zeta}{B_s(\log\Delta_\varphi)\zeta}+o(\varepsilon^2).
\end{equation}

For $0<s<1$, use the positive Haagerup quasi-norm from the same formula~\eqref{main:norm} and extend the hyperbolic definition by $s'=s/(s-1)$ and $B_s(0)=s-1$. 

At $s=1$ put $B_1=0$; the positive-input norm is exactly $\varphi(1+\varepsilon a)=1$.
\end{lemma}
\begin{proof}
If $a=0$, the lemma holds. 
We assume that $a\neq 0$ in the following.
We use one uniform Taylor estimate for all $s>0$. First work in a finite
tracial von Neumann algebra $(\mathcal A,\tau)$, with $D\in \mathcal A$ strictly positive,  $D^{-1}$ bounded, $\tau(D)=1$, and $\varphi(x)=\tau(Dx)$. Put
\[
F(z)=\tau\left(\left(D^{1/s}+z\bj_D^{1/s}(a)\right)^s\right).
\]
For $|z|\norm a<1$,
$\operatorname{Re}(D^{1/s}+z\bj_D^{1/s}(a))\ge(1-|z|\norm a)D^{1/s}>0$, so the principal power and
$F$ are holomorphic. 
By~Inequalities \eqref{eq:logremainder0}, 
\[
 L(z)=\log(D^{1/s}+z\bj_D^{1/s}(a))-\log D^{1/s},
 \qquad \norm{L(z)}\le\frac{|z|\norm a}{1-|z|\norm a}.
\]
For every bounded, possibly non-self-adjoint $v$, the Duhamel expansion and
trace H\"older inequality give
\[
 \left|\tau (e^{\log D+v})\right|
 \le\sum_{n=0}^\infty\frac{\norm v^n}{n!}\tau(D)
 =e^{\norm v}.
\]
Since $F(z)=\tau \left(e^{\log D+sL(z)}\right)$, it follows that $|F(z)|\le e^s$
on $\displaystyle |z|=\frac{1}{2\norm a}$. 
Hence for $|\varepsilon|\norm a\le1/4$,
\begin{equation}\label{eq:relativeTaylorNorm}
 \left|F(\varepsilon)-F(0)-\varepsilon F'(0)
                   -\frac{\varepsilon^2}{2}F''(0)\right|
 \le16e^s\norm a^3|\varepsilon|^3.
\end{equation}
The constant is independent of the norm of $D$.

It remains to identify the quadratic coefficient. The trace derivative and
the divided difference of $x^{s-1}$ give
\[
 F(0)=1,\qquad F'(0)=s\varphi(a)=0,\qquad
 F''(0)=s\ip\zeta{B_s(\log\Delta_\varphi)\zeta}.
\]
Indeed, for spectral variables $d_j,d_k>0$ of $D$, changing from
$D^{1/(2s)}aD^{1/(2s)}$ to $D^{1/4}aD^{1/4}$ gives the symbol
\[
 (d_jd_k)^{1/s-1/2}
 \frac{d_j^{1-1/s}-d_k^{1-1/s}}{d_j^{1/s}-d_k^{1/s}}
 =\frac{\sinh((1-1/s)\log(d_j/d_k)/2)}
        {\sinh(\log(d_j/d_k)/(2s))}
 =B_s(\log(d_j/d_k)).
\]
The diagonal value is $s-1$. This calculation first applies to finite
spectral decompositions of $D$; 
then approximation in operator norm by spectral step functions proves it for bounded invertible $D$ in a finite von Neumann algebra.
Thus~\eqref{eq:relativeTaylorNorm}, followed by taking the power $2/s$,
proves~\eqref{eq:hessian} in this setting.

For a general $(\cM,\varphi)$, use the state-preserving Haagerup reduction $\cM\subset (\mathcal R, \widehat \varphi)$ and expectations $\mathbb E_n:\mathcal R\to\mathcal R_n$ onto increasing finite von Neumann algebras~\cite[Theorems 2.1 and 3.1]{HJX}.
The restricted reference states have bounded invertible densities relative to the finite traces. 
Set $a_n=\mathbb E_n(a)$. 
These elements are self-adjoint, centered, uniformly bounded by $\norm a$, and have modular spectral support in one fixed compact interval, because $\mathbb  E_n$ commutes with the modular group.
We have that $\bj_{\widehat \varphi}^{1/2}(a_n-a) \to 0$. 
The H\"older inequality and interpolation theorem~\cite{Kos84}, respectively, imply that
\[
 \norm{a_n-a}_{s,\widehat\varphi}
 \le
 \begin{cases}
  \norm{a_n-a}_{2,\widehat\varphi},&0<s\le2,\\[2pt]
  (2\norm a)^{1-2/s}\norm{a_n-a}_{2,\widehat\varphi}^{2/s},&s>2.
 \end{cases}
\]
Hence the $s$-norms (or quasi-norms) converge at every fixed
$\varepsilon$. 
The vectors $\bj_{\widehat \varphi}^{1/2}(a_n)$ converge in $L^2$, and
$B_s(\log\Delta_{\widehat\varphi})$ is bounded on their common compact
spectral support, so the quadratic coefficients converge as well.
The uniform remainder~\eqref{eq:relativeTaylorNorm} passes the whole
Taylor expansion to the limit. 
The expected inclusion of $\cM$ is isometric in every Haagerup $L^s(\cM)$ and intertwines the modular groups, which proves the lemma in $\cM$.
\end{proof}

For example, if $a=g+g^*$, $\sigma_t^\varphi(g)=\lambda^{it}g$ with $\lambda>1$, then $\varphi(gg^*)=\lambda \varphi(g^*g)$ and $\norm\zeta^2=2\sqrt\lambda \varphi(g^*g)$. 
Hence
\begin{equation}\label{eq:eigenhessian}
 \norm{1+\varepsilon(g+g^*)}_{s,\varphi}
 =1+\sqrt\lambda \varphi(g^*g)\,B_s(\log\lambda)\varepsilon^2+o(\varepsilon^2).
\end{equation}
Only second moments occur, which explains the absence of $q$ from the sharp obstruction.

Assume $\cH_{\RR}\ne\{0\}$, and recall that $J_{c}$ is the real conjugation on $\cH_{\CC}$. 
Write $\widetilde{A}=\log A$ and let $E_{\widetilde{A}}$ be its spectral measure.  
The identity $J_{c}AJ_{c}=A^{-1}$ gives
$J_{c} \widetilde{A}J_{c}=-\widetilde{A}$, hence
\[
 J_{c} E_{\widetilde{A}}(I)J_{c}=E_{\widetilde{A}}(-I)
\]
for every Borel set $I$ in $\RR$.  
In particular, the projection associated with $I\cup(-I)$ preserves $\cH_{\RR}$.
On every compact spectral interval the original Hilbert norm and the
deformed Hilbert norm are equivalent, so these spectral vectors are
legitimate vectors of the $q$-Fock one-particle space.

Suppose first that $0<\beta=\log\|A\|<\infty$.
For $0<\delta<\beta$, the projection
\[
 E_\delta=E_{\widetilde{A}}([\beta-\delta,\beta]\cup[-\beta,-\beta+\delta])
\]
is nonzero and preserves the real subspace. 
More explicitly, choose
$0\ne v_\delta\in E_{\widetilde{A}}([\beta-\delta,\beta])\cH_{\CC}$ and set
$\xi_\delta=v_\delta+ J_c v_\delta$. 
The summands lie in disjoint spectral bands, so $\xi_\delta\ne0$, and $J_c\xi_\delta=\xi_\delta$. 
Put
\[
 a_\delta=s_q(\xi_\delta)
 =\ell_q(\xi_\delta)+\ell_q(\xi_\delta)^*.
\]
This is a bounded, nonzero, centered, self-adjoint first-chaos element.
With the convention $U_t=A^{it}$, its modular action is
$\sigma_t^\varphi(s_q(\xi))=s_q(U_{-t}\xi)$, and the one-particle
modular operator is $A^{-1}$.  
Thus
\[
 \zeta_\delta=\Delta_\varphi^{1/4}a_\delta\Omega
 =A^{-1/4}\xi_\delta
\]
has the same two compact spectral bands for $\log\Delta_\varphi$.
In particular $a_\delta$ is modular-analytic.  One may see this directly
from its entire extension
\[
 z\longmapsto
 \ell_q(e^{-iz{\widetilde{A}}}\xi_\delta)
 +\ell_q(e^{-i\overline z{\widetilde{A}}}\xi_\delta)^*.
\]
Consequently, for every continuous even function $F$ on $[-\beta,\beta]$,
\[
 \frac{\langle\zeta_\delta,
       F(\log\Delta_\varphi)\zeta_\delta\rangle}
      {\|\zeta_\delta\|_2^2}
 \longrightarrow F(\beta).
\]

If $\beta=0$, then $A=1$ and any nonzero real vector gives a
self-adjoint first-chaos element with modular frequency zero.
If $\log A$ is unbounded, its spectrum is unbounded in both directions.
Choose compact intervals $I_n\subset[n,\infty)$ with $E_{\widetilde{A}}(I_n)\ne0$
and nonzero real vectors in
$E_{\widetilde{A}}(I_n\cup(-I_n))\cH_{\RR}$.  The associated fields have
compact modular spectral support contained in $\{|v|\ge n\}$.

Let
\begin{equation}\label{same:threshold}
 \tau(p,r)=\inf\big\{t\geq0:
       \|\Phi_t\|_{  p\to  r}\leq1\big\}.
\end{equation}
\begin{proposition}\label{prop:necessary}
 Suppose that $\{\Phi_t\}_{t\geq 0}$ is a $q$-Araki-Woods Ornstein-Uhlenbeck Semigroup with $\beta=\log\|A\|$.
 For $1< p< r< \infty$, 
 \begin{itemize}
 \item if $\beta<\infty$, then
 \begin{align*}
     \tau(p, r) \geq \frac{1}{2}\log\frac{B_r(\beta)}{B_p(\beta)},
 \end{align*}
\item if $\beta=\infty$, no finite hypercontractive time exists for any such pair.
\end{itemize}
\end{proposition}
\begin{proof}
Fix a time $t$ for which the norm inequality holds and apply it to
$1+\varepsilon a_\delta$ from the preceding construction.
Since $\Phi_t(a_\delta)=e^{-t}a_\delta$, Lemma~\ref{lem:hessian} gives
\[
 e^{-2t}\langle\zeta_\delta,
           B_r(\log\Delta_\varphi)\zeta_\delta\rangle
 \le \langle\zeta_\delta,
           B_p(\log\Delta_\varphi)\zeta_\delta\rangle.
\]
For $\beta>0$, divide by the positive denominator and let
$\delta\downarrow0$. Both $B_p$ and $B_r$ are continuous and even,
so $e^{-2t}\le B_p(\beta)/B_r(\beta)$.
For $\beta=0$, any nonzero real first-chaos field gives
$e^{-2t}\le(p-1)/(r-1)$. Taking the infimum over admissible times
proves the assertion, without assuming that the infimum is attained.

For unbounded $\log A$, use the compact bands with absolute frequencies at least $n$. 
The quotient $B_r(v)/B_p(v)$ is even, nondecreasing for $v\ge0$ by Lemma~\ref{lem:monotone}, and tends to infinity. 
Thus the necessary lower bounds on $t$ diverge with $n$.

\end{proof}



\subsection{Sharp \texorpdfstring{$2$}{2}-LSI and Hypercontractivity}
\label{main:gross}
In this section, we shall deduce the $2$-logarithmic Sobolev inequality and the hypercontractivity.

\begin{lemma}\label{lem:entTaylor}
Suppose that $\cM$ is a von Neumann algebra with a faithful normal state $\varphi$ and $D_\varphi$ is its Haagerup density.
Let $a=a^*\in \cM$ with $\varphi(a)=0$. 
For real \(\varepsilon\) satisfying \(|\varepsilon|\|a\|<1\), let $\rho_\varepsilon$ be the state with density $D_\varphi^{1/2}(1+\varepsilon a)D_\varphi^{1/2}$. 
Put $\xi_a=D_\varphi^{1/4}a D_\varphi^{1/4}$ and
\[
 \vartheta(v)=\frac{v}{2\sinh(v/2)},\qquad \vartheta(0)=1.
\]
Then
\begin{equation}\label{eq:entTaylor}
 D(\rho_\varepsilon\Vert\varphi)
 =\frac{\varepsilon^2}{2}
   \ip{\xi_a}{\vartheta(\log\Delta_\varphi)\xi_a}+O(|\varepsilon|^3).
\end{equation}

\end{lemma}
\begin{proof}
Use the estimate~\eqref{eq:logremainder0} with $y=D_\varphi$.
The derivative map $\mathbf K_{D_\varphi}$ has the scalar spectral kernel
\[
 \sqrt{d_j d_k}\int_0^\infty\frac{dt}{(d_j+t)(d_k+t)}
 =\frac{\sqrt{d_j d_k}(\log d_j-\log d_k)}{d_j-d_k}
 =\vartheta(\log(d_j/d_k)).
\]
Equivalently, the Fourier transform of $(\pi/2)\sech^2(\pi t)$ gives
\[
\mathbf K_{D_\varphi}(x)
 =\frac\pi2\int_{\mathbb R}\sech^2(\pi t)
                  \sigma_t^\varphi(x)\,dt.
\]
Thus $\mathbf K_{D_\varphi}$ is a normal unital completely positive map on $\cM$, with symmetric $L^2$
implementation $\vartheta(\log\Delta_\varphi)$. 
For self-adjoint $x$ with $\norm x<1$, put $D_x=D_\varphi^{1/2}(1+x)D_\varphi^{1/2}$. 
Its bounded logarithmic difference $\Lambda_{D_\varphi}(x)=\log D_x-\log D_\varphi$ belongs to $\cM$: 
the dual action scales both densities and hence fixes their logarithmic difference.

Let $\omega_a$ be the normal functional with density $D_\varphi^{1/2}a D_\varphi^{1/2}$,
so $\norm{\omega_a}\le\norm a$, and put $G(\varepsilon)=D(\rho_\varepsilon\Vert\varphi)$.
The path is $C^1$ in predual norm and uniformly comparable to $\varphi$.
The entropy derivative formula~\cite[Proposition 2.5]{Wir25} therefore gives
\[
 G'(\varepsilon)=\omega_a(\Lambda_{D_\varphi}(\varepsilon a)).
\]
By~\eqref{eq:logremainder0}, whenever
$|\varepsilon|\norm a\le r_0<1$,
\[
 \left|G'(\varepsilon)
       -\varepsilon\omega_a(\mathbf K_{D_\varphi}(a))\right|
 \le\frac{\varepsilon^2\norm a^3}{1-r_0}.
\]
Note that $\omega_a(\mathbf K_{D_\varphi}(a))
 =\ip{\xi_a}{\vartheta(\log\Delta_\varphi)\xi_a}.$
By integrating from $0$ to $\varepsilon$, this proves the lemma.
\end{proof}

\begin{theorem}[Sharp \(2\)-LSI]\label{main:sharp-lsi2}
Suppose that $\{\Phi_t\}_{t\geq 0}$ is a $q$-Araki-Woods Ornstein-Uhlenbeck Semigroup with $\beta=\log \|A\|$ and $\cH_{\RR}\neq 0$.
We have that
\begin{equation}\label{main:sharp-alpha2}
 \alpha_2(\cL)=\left\{\begin{array}{ll}
 \frac{1}{2}, & \beta=0 \\
 \frac{2\tanh(\beta/4)}{\beta},  & 0<\beta< \infty \\
 0 , & \beta=\infty.
 \end{array}\right.
\end{equation}
\end{theorem}

\begin{proof}
First suppose that $\beta<\infty$.
Let $\displaystyle a_\beta=\frac{2\tanh(\beta/4)}{\beta}$.
We first establish the lower bound. 
Let $x$ be a strictly positive finite Wick polynomial with modular-analytic letters, and let
\[
 t(r)=\frac12\log B_r(\beta),\qquad t(2)=0.
\]
On the finite-chaos decomposition of $x$, define $x_r=\Phi_{-t(r)}(x)$ by multiplying its degree-$k$ component by $e^{kt(r)}$. 
This finite sum converges to $x$ in operator norm as $r\downarrow 2$. 
Applying Theorem~\ref{thm:cross2r} to $x_r$ gives $\norm x_{r,\varphi}\le\norm{\Phi_{-t(r)}(x)}_{2,\varphi}.$
The two sides agree at $r=2$. 
Differentiating at $r=2^+$, we see that
\[
 \frac{\Ent_{2,\varphi}(x)}{4\norm x_{2,\varphi}^2}
 \le t'(2)\frac{\mathcal E_{2,\varphi}(x)}{\norm x_{2,\varphi}^2}.
\]
The derivative formula for $B_r$ gives
\[
 4t'(2)=\frac{\beta}{2\tanh(\beta/4)}=a_\beta^{-1},
\]
including its limit at $\beta=0$. 
Adding $\epsilon 1$ with $\epsilon >0$ and letting $\epsilon\downarrow0$ handles positive semidefinite Wick polynomials by the entropy lower semicontinuity.

Let $\mathscr P$ be the $*$-algebra of Wick polynomials whose letters have compact $\log A$ spectral support. 
It is modular-analytic and strongly dense. 
For bounded $x\ge0$, choose uniformly bounded $p_j\in\mathscr P$ with $p_j\to x^{1/2}$ strongly*. 
Then $p_j^*p_j\to x$ strongly*, hence in the symmetric $L^2$ embedding. For each fixed $s>0$,
\[
 \bj_{\varphi}^{1/2}(\Phi_s(p_j^*p_j))\longrightarrow e^{-s\cL}\bj_{\varphi}^{1/2}(x)
\]
in the norm $\norm{\cdot}_{\cL}$, because $\cL^{1/2}e^{-s\cL}$ is bounded. 
If $\bj_{\varphi}^{1/2}(x)\in \dom \cL^{1/2}$, the limit $s\downarrow 0$ also holds in the norm $\norm{\cdot}_{\cL}$. 
For the entropy limit, we use
$\norm{y_j^2-y^2}_1\le(\norm{y_j}_2+\norm y_2)\norm{y_j-y}_2$
and lower semicontinuity of the Araki entropy. 
For completeness, let $y\ge0$ be any vector in $\dom \cL^{1/2}$. 
Positive polynomial vectors $\bj_{\varphi}^{1/2}(p^*p)$ are dense in the positive cone: first approximate by $\bj_{\varphi}^{1/2}(\cM_+)$, then use the preceding square-root approximation. 
Choose $s_n\downarrow0$, $\eta_n\downarrow0$, and such vectors $z_n$ satisfying
\[
 \norm{z_n-y}_2\le
 \frac{\eta_n}{1+\norm{\cL^{1/2}e^{-s_n\cL}}}.
\]
Then $y_n=e^{-s_n\cL}z_n$ is again a positive modular-analytic polynomial vector and, using the equivalent norm $\norm v_2+\norm{\cL^{1/2}(v)}_2$,
\[
 \norm{y_n-y}_{\cL}
 \le\eta_n+\norm{e^{-s_n\cL}y-y}_{\cL}\longrightarrow0.
\]
Entropy lower semicontinuity now proves the lower bound on the entire positive closed form domain.

For optimality, put $x_\varepsilon=1+\varepsilon a$, where $a$ is a nonzero centered self-adjoint first-chaos element with compact modular spectral support, and set $\xi=\Delta_\varphi^{1/4}a\Omega$.
We see that $\mathcal E_{2,\varphi}(x_\varepsilon)=\varepsilon^2\norm\xi^2.$
Applying Lemma~\ref{lem:entTaylor} to 
the following perturbation
\begin{align*}
  \frac{1}{\left\|\bj^{1/2}_\varphi(x_{\varepsilon})\right\|_{2}^2} (\bj_\varphi^{1/2}(x_\varepsilon))^2 =D_\varphi^{1/2}
   \frac{1+\varepsilon \widetilde{a}+\varepsilon^2 \sigma_{i/4}^\varphi(a)\sigma_{-i/4}^\varphi(a)}{1+\varepsilon^2 \|\xi\|_2^2}
  D_\varphi^{1/2}
\end{align*}
where
$\widetilde{a}=\sigma_{i/4}^\varphi(a)+\sigma_{-i/4}^\varphi(a)$, we obtain that
\begin{equation}\label{eq:entHessian}
\begin{aligned}
 \Ent_{2,\varphi}(x_\varepsilon)
 = \|x_{\varepsilon}\|_{2, \varphi}^2 D(\rho_{x_\varepsilon}\|\varphi)
 =\frac{\varepsilon^2}{2}
   \ip{\xi_{\widetilde{a}}}{\vartheta(\log\Delta_\varphi)\xi_{\widetilde{a}}}+O(|\varepsilon|^3)\\
 =\varepsilon^2\ip\xi{\left(\frac{1}{2}\log \Delta_\varphi \right)\coth\left(\frac{1}{4}\log \Delta_\varphi\right)\xi}+O(|\varepsilon|^3),
 \end{aligned}
\end{equation}
where
\[
 \xi_{\widetilde{a}}=D_\varphi^{1/4}\widetilde{a} D_\varphi^{1/4}
 =\bigl(\Delta_\varphi^{1/4}+\Delta_\varphi^{-1/4}\bigr)\xi
 =2\cosh\!\left(\frac14\log\Delta_\varphi\right)\xi.
\]
Note that the normalized perturbation is $\displaystyle a_{\varepsilon}=\frac{\varepsilon \widetilde{a}+\varepsilon^2 (\sigma_{i/4}^{\varphi}(a)\sigma_{-i/4}^{\varphi}(a)-\|\xi\|_2^2)}{1+\|\xi\|_2^2\varepsilon^2}$ satisfying $a_\varepsilon^*=a_{\varepsilon}$, $\varphi(a_\varepsilon)=0$ and $a_\varepsilon= \varepsilon \widetilde{a}+O(\varepsilon^2)$.

Define the even continuous function
\[
 F(v)=\frac v2\coth(v/4),\qquad F(0)=2.
\]
The logarithmic Sobolev constant $\alpha_2(\cL)$ therefore obeys
\[
 \alpha_2(\cL)\le
 \frac{\|\xi\|_2^2}
      {\langle\xi,F(\log\Delta_\varphi)\xi\rangle}.
\]
Use the paired spectral bands constructed before
Proposition~\ref{prop:necessary}. For $\beta>0$ the quotient tends
to $1/F(\beta)=2\tanh(\beta/4)/\beta$; for $\beta=0$ it is $1/2$.
No endpoint eigenvector is assumed. If the modular spectrum is
unbounded, the same quotient is at most
$\displaystyle \frac{1}{\inf_{|v|\ge n}F(v)}$ on the bands at frequencies at least $n$,
and hence tends to zero. Thus no positive $2$-LSI constant exists
in that case.
\end{proof}

\begin{corollary}[Hypercontractivity along Gross Curve]\label{cor:HC}
Suppose that $\{\Phi_t\}_{t\geq 0}$ is a $q$-Araki-Woods Ornstein-Uhlenbeck semigroup and $\beta=\log \|A\|<\infty$.
Let $\displaystyle \left.\frac{\beta}{8\tanh(\beta/4)}\right|_{\beta=0}:=\frac{1}{2}$.
For every \(1<p<r<\infty\),
\begin{equation}\label{main:gross-curve}
 \|\Phi_t(x)\|_{r,\ph}\le\|x\|_{p,\ph}
 \quad\text{whenever}\quad
 t\ge \frac{\beta}{8\tanh(\beta/4)}\log\frac{r-1}{p-1}.
\end{equation}
The coefficient is optimal among the Gross curves.
\end{corollary}
\begin{proof}
Differentiate the uniform spin estimate of Proposition ~\ref{spin:finiteHC}, namely \eqref{spin:finitecontraction} under \eqref{spin:time} at $r=2$. 
This gives $\alpha_2(\cL_n^{\varepsilon})\ge 2\tanh(\beta/4)/\beta$. 
The semigroup is GNS-symmetric. 
By the finite-dimensional Stroock--Varopoulos inequality~\cite[Theorem 14]{BDR}, we have that $\alpha_s(\cL_n^{\varepsilon})\ge 2\tanh(\beta/4)/\beta$ for $s>1$. 

Hence Equation \eqref{main:gross-curve} holds for $x\geq 0$. 
By Lemma~\ref{hc:positive-extension}, the estimate extends to arbitrary inputs. 
The matricial approximation of Section~\ref{spin:CLTsection} transfers the resulting uniform norm inequality to $\cM$.
The sharp value of $\alpha_2$ implies that a smaller coefficient cannot work for all pairs, by differentiating a purported improved $2\to r$ Gross curve at $r=2$.

\end{proof}

\begin{theorem}
Suppose that $\{\Phi_t\}_{t\geq 0}$ is a $q$-Araki-Woods Ornstein-Uhlenbeck semigroup with $\beta=\log \|A\|<\infty$.
For $ 1<p\leq2\leq r<\infty,$
\begin{equation}
 \|\Phi_t\|_{p\to r}\leq1
 \quad\text{ if and only if }\quad
 e^{-2t}\leq\frac{B_p(\beta)}{B_r(\beta)}.
 \qquad
 \label{spin:crossing}
\end{equation}
\end{theorem}
\begin{proof}
The case $p=r=2$ is ordinary $L^2$ contractivity. 
For every remaining pair, KMS symmetry and the weighted dual pairing give
$\norm{\Phi_t}_{p\to2}=\norm{\Phi_t}_{2\to p'}$. 
By Theorem~\ref{thm:cross2r}, the first map is contractive at time $-\frac12\log B_p(\beta)$, since $B_{p'}=B_p^{-1}$. 
The map $2\to r$ is contractive at time $\frac12\log B_r(\beta)$. 
Compose these maps and use the semigroup property. 
The necessity follows from Proposition~\ref{prop:necessary}.

\end{proof}

\begin{corollary}\label{same:below-two}
Suppose that $\{\Phi_t\}_{t\geq 0}$ is a $q$-Araki-Woods Ornstein-Uhlenbeck semigroup with $\beta=\log \|A\|<\infty$.
For \(1<p<r<2\),
\begin{equation}\label{same:below-bracket}
 \frac12\log\frac{B_r(\beta)}{B_p(\beta)}
 \leq\tau(p,r)\leq
 \min\left\{
 \frac{\beta}{8\tanh(\beta/4)}\log\frac{r-1}{p-1},\,
 \frac12\log B_{\,2p'/r'}(\beta)
 \right\}.
\end{equation}
For \(2<p<r<\infty\),
\begin{equation}\label{same:below-bracket2}
 \frac12\log\frac{B_r(\beta)}{B_p(\beta)}
 \leq\tau(p,r)\leq
 \min\left\{
 \frac{\beta}{8\tanh(\beta/4)}\log\frac{r-1}{p-1},\,
 \frac12\log B_{\,2r/p}(\beta)
 \right\}.
\end{equation}
\end{corollary}

\begin{proof}
For $2<p<r$, take $s=2r/p>2$ and a time at which $\Phi_t: L^2(\cM)\to L^s(\cM)$ is contractive. 
Interpolate with the contraction $L^\infty(\cM)\to L^\infty(\cM)$, using weight $\theta=2/p$. 
The domain reciprocal exponent is $\theta/2=1/p$ and the range reciprocal exponent is $\theta/s=1/r$. 
For $p<r<2$, first dualize to $r'\to p'$, where $2<r'<p'$, and use the preceding calculation. 
\end{proof}

These estimates suggest the following sharp hypercontractivity conjecture.
\begin{conjecture}\label{con:hyper}
Suppose that $\{\Phi_t\}_{t\geq 0}$ is a $q$-Araki-Woods Ornstein-Uhlenbeck semigroup with $\|A\|<\infty$.
For $ 1<p\leq r<\infty,$
    \begin{equation}
 \|\Phi_t\|_{p\to r}\leq1
 \quad\text{ if and only if }\quad
 e^{-2t}\leq \frac{B_p(\beta)}{B_r(\beta)}.
 \qquad
 \label{spin:crossing0}
\end{equation}
\end{conjecture}

\begin{remark}
The conjecture is proved here when $p\le2\le r$, and in the tracial case by the sharp Gross curve.
Suppose that for every finite signed-spin system with $1\leq \mu_j^4 \leq e^\beta$ as above, every choice of its commutation signs, every \(1<s<\infty\), and every positive invertible \(x\in \Gamma_{\langle 1, \ldots, n\rangle}\),
\begin{equation}\label{same:spin-lsi}
 \Energ_{s,\tau_n^{\varepsilon}}(x)
 \geq \alpha_s^*(\beta)\Ent_{s,\tau_n^{\varepsilon}}(x),
\end{equation}
where
\begin{equation}\label{same:alpha}
 \alpha_s^*(\beta)=
 \frac{s s'\sinh(\beta/(2s))\sinh(\beta/(2s'))}
      {\beta\sinh(\beta/2)},
 \qquad \alpha_s^*(0)=\frac12.
\end{equation} 
Then Conjecture \ref{con:hyper} holds.
However, a proof for one qubit is insufficient to establish~\eqref{same:spin-lsi} for arbitrary signs.
\end{remark}

\section{Complete Modified Logarithmic Sobolev Inequalities}\label{sec:cmlsi}

In this section, we shall obtain the complete modified logarithmic Sobolev inequalities for the $q$-Araki-Woods Ornstein-Uhlenbeck semigroups via the intertwining property.
It is worth to point out that Lee, Xu, Youn, and Zhang \cite{LXYZ} already construct a gradient through rotations on a doubled \(q\)-Araki–Woods space.

Firstly, let us recall the complete modified logarithmic Sobolev inequalities (CMLSI) and complete Fisher monotonicity (CFM).
Suppose that $\{\Phi_t\}_{t\geq 0}$ is a GNS-symmetric quantum Markov semigroup with generator $\cL$ and $\kappa\in \RR$.
Suppose that $\displaystyle \mathbb E_{\Phi}=\lim_{t\to\infty} \Phi_t$.
For a regular normal positive functional $\rho$, the Fisher information $\cI_{\cL}$ is defined to be 
\begin{align*}
\cI_{\cL}(\rho\|  \mathbb E_{\Phi}^*(\rho))=tr(\cL^*(D_\rho)(\log D_\rho-\log D_{ \mathbb E_{\Phi}^*(\rho)})),
\end{align*}
where $D_\rho, D_{ \mathbb E_{\Phi}^*(\rho)}$ are the Haagerup densities.
Here regular means that $\rho\in\dom \cL_*$ and $\rho$ is bounded above and below by positive multiples of the faithful functional
$\mathbb E_{\Phi}^*\rho$. 
These conditions make the logarithmic difference a
bounded element of the algebra.
We refer to \cite{Wir25} for the details on the Fisher information.

The semigroup $\{\Phi_t\}_{t\geq 0}$ satisfies $MLSI(\kappa)$ if 
\begin{align*}
\cI_{\cL}(\rho\| \mathbb E_{\Phi}^*(\rho))\geq \kappa D(\rho\| \mathbb E_{\Phi}^*(\rho));
\end{align*}
satisfies $\cmlsi(\kappa)$ if $\Phi_t\otimes \Id_{\cN}$ satisfies $MLSI(\kappa)$ for every $\sigma$-finite von Neumann algebra $\cN$.
The semigroup $\{\Phi_t\}_{t\geq 0}$ satisfies $FM(\kappa)$ if 
\begin{align*}
\cI_{\cL}(\Phi_t^*(\rho)\| \mathbb E_{\Phi}^*(\rho))\leq e^{-t\kappa}\cI_{\cL}(\rho\| \mathbb E_{\Phi}^*(\rho)).
\end{align*}
It satisfies $\cfm(\kappa)$ if $\Phi_t\otimes \Id_{\cN}$ satisfies $FM(\kappa)$ for every $\sigma$-finite von Neumann algebra $\cN$.

Suppose that 
\[
 \widehat \Gamma_q=\Gamma_q(\cH_{\RR}\oplus \cH_{\RR},U\oplus U),
\]
with vacuum state \(\widehat\ph\), and embed \(\Gamma_q\) into its first copy by the embedding map \(j\). 
The projection onto the first copy gives a normal vacuum-preserving conditional expectation \(\mathbb E: \widehat \Gamma_q\to j(\Gamma_q)\).
Let
\[
 R_\theta=
 \begin{pmatrix}
 \cos\theta I&-\sin\theta I\\
 \sin\theta I&\cos\theta I
 \end{pmatrix}.
\]
Since \(R_\theta\) commutes with \(U_s\oplus U_s\) for all $s\in \RR$, its second quantization defines vacuum-preserving automorphisms \(\alpha_\theta\) of \(\widehat \Gamma_q\).
On Wick polynomials, we define a linear map $\partial $ into $\widehat \Gamma_q$:
\begin{equation}\label{eq:rotation-derivation}
 \partial x=\left.\frac{d}{d\theta}\alpha_\theta(j(x))\right|_{\theta=0}.
\end{equation}
Differentiating multiplicativity and $*$-preservation gives
\[
 \partial(xy)=\partial x\,j(y)+j(x)\partial y,\qquad
 \partial(x^*)=(\partial x)^*.
\]
It also commutes with the modular automorphism groups.

For a degree-$k$ Wick operator $x_k$, we have that
\begin{equation}\label{eq:compressionrotation}
\mathbb E(\alpha_\theta j(x_k))=(\cos\theta)^k j(x_k).
\end{equation}
Let $\displaystyle x=\sum_kx_k$ be its finite chaos decomposition. On GNS vectors the automorphism is implemented by the Fock unitary $\mathcal F_q(R_\theta)$, and
\[
 \ip{j(x)\widehat\Omega}{\mathcal F_q(R_\theta)jx\widehat\Omega}
 =\sum_k(\cos\theta)^k\norm{x_k\Omega}^2.
\]
Twice differentiating at zero, and using that the infinitesimal generator of a unitary group is skew-adjoint, we obtain that 
\begin{equation}\label{eq:GNSderivation}
 \norm{\partial x\,\widehat\Omega}^2
 =\widehat\varphi((\partial x)^*\partial x)
 =\sum_k k\norm{x_k\Omega}^2=\varphi(x^*\cL(x)).
\end{equation}
Let $\mathcal P$ be the analytic Wick algebra with compact one-particle modular support. 
The vectors $\mathcal P\Omega$ are a core for $\cL^{1/2}$: truncate the total degree and then approximate each tensor letter
by compact spectral vectors. 
Initially define
$\delta_0(xD_\varphi^{1/2})=(\partial x)D_{\widehat\varphi}^{1/2}$
for $x\in\mathcal P$.
We have that $\delta_0$ is closable, and its closure $\delta$ has domain exactly $\dom \cL^{1/2}$. 
Because it commutes with the modular groups, its symmetric version is
\begin{equation}\label{eq:symmetricderivation}
 \norm{D_{\widehat\varphi}^{1/4}(\partial x)D_{\widehat\varphi}^{1/4}}_2^2
 =\mathcal E_{2,\varphi}(x).
\end{equation}


Let $\cH_\delta$ be the closed $\Gamma_q$-bimodule generated by the range of $\delta$ inside $L^2(\widehat \Gamma_q)$. 
It is reducing for both normal actions. 
Let
\begin{align*}
    \mathscr{J}_{\delta}=J_{\widehat\varphi}|_{\cH_\delta}, \quad \mathscr{U}_{\delta, s}=\Delta_{\widehat \varphi}^{is} |_{\cH_\delta}.
\end{align*}
Reality and modular covariance on the analytic core, extended by closure, imply that \(\mathcal H_\delta\) is invariant under \(J_{\widehat\varphi}\) and \(\Delta_{\widehat\varphi}^{is}\). 
Hence the following restrictions are well defined.
\begin{align*}
    \delta J_\varphi =\mathscr{J}_{\delta} \delta, \quad \delta\Delta_\varphi^{is}=\mathscr{U}_{\delta, s}\delta.
\end{align*}
The Leibniz rule is
\[
 \delta(xD_\varphi^{1/2}y)
 =j(x)\delta(D_\varphi^{1/2}y)
  +\delta(xD_\varphi^{1/2})j(y).
\]
The algebraic Leibniz rule, reality, modular covariance, and~\eqref{eq:GNSderivation} identify it with the first-order calculus of the number form. In particular the semigroup is $\Gamma$-regular in the terminology of~\cite{Wir25}. 
Equivalently, its carr\'e du champ is
\[
 j(\Gamma(x,y))=\frac12 j\bigl(\cL(x)^*y+x^*\cL(y)-\cL(x^*y)\bigr)
 =\bE((\partial x)^*\partial y),
\]
as follows by differentiating the product twice in~\eqref{eq:compressionrotation}.

Define the auxiliary semigroup and its GNS implementation by
\[
 \Psi_t=\Gamma_q(e^{-t}I\oplus I),\qquad Q_t(z\widehat\Omega)=\Psi_t(z)\widehat\Omega.
\]
The map $\Psi_t$ is unital completely positive, state-preserving, GNS-symmetric, and modularly covariant; moreover $\Psi_tj=j\Phi_t$. On a degree-$k$ word, every summand of $\partial x$ has $k-1$ first-copy letters and one second-copy letter. Thus
\begin{equation}\label{eq:intertwining}
 \partial\Phi_t=e^{-t}\Psi_t\partial.
\end{equation}
The identity extends from the core by closure. If $P_\delta$ is the projection onto $\cH_\delta$ and $T_t$ is the GNS implementation of $\Phi_t$, define
\[
 \vec T_t=e^{-t}P_\delta Q_t|_{\cH_\delta}.
\]
Then $\delta T_t=\vec T_t\delta$ on the form domain. This compression avoids an unnecessary assumption that the auxiliary map preserves the whole minimal bimodule. Wirth's criterion requires a family of bounded operators here, not a semigroup law for $\vec T_t$.


\begin{lemma}\label{lem:cporder}
Let $\Phi:\cB\to \cB$ be normal, unital completely positive, state-preserving and GNS-symmetric for a faithful normal state $\omega$. 
Let $T$ be its GNS implementation, and let $\pi_l,\pi_r$ be the standard left and right actions. 
Then
\begin{equation}\label{eq:cporder}
 T^*\pi_l(\cdot)T\le_{\mathrm{cp}}\pi_l\circ \Phi,
 \qquad T^*\pi_r(\cdot)T\le_{\mathrm{cp}}\pi_r\circ \Phi.
\end{equation}
The right-action inequality is understood on the opposite algebra.
\end{lemma}
\begin{proof}
Take a Stinespring representation $\pi_l(\Phi(a))=V^*\pi(a)V$, where $\pi$ is a $*$-homomorphism of $\cB$ and $V$ is a contraction. 
The map $W$ defined as 
\[
 W(x\Omega)=\pi(x)V\Omega
\]
is isometric, since
$\norm{\pi(x)V\Omega}^2=\omega(\Phi(x^*x))=\omega(x^*x)$, where $\Omega$ is the cyclic and separating vector for $\omega$. 
It intertwines the left actions. 
GNS symmetry gives, on the dense GNS core,
\[
 \ip{W^*Vy\Omega}{x\Omega}
 =\omega(\Phi(x^*)y)=\omega(x^*\Phi(y))
 =\ip{Ty\Omega}{x\Omega}.
\]
Hence $W^*V=T$. 
The projection $P=WW^*$ commutes with $\pi(\cB)$, because the range of $W$ reduces $\pi(\cB)$. 
Consequently
\[
 \pi_l(\Phi(a))-T^*\pi_l(a)T
 =V^*\pi(a)(1-P)V.
\]
This is a completely positive map of $a$: it is a Stinespring compression of the representation on $(1-P)\mathcal K$. 
It proves the left inequality with all matrix amplifications. 
The GNS implementation of a GNS-symmetric Markov map commutes with modular conjugation. 
Conjugating the left inequality by that conjugation gives the right inequality on $\cB^{\mathrm{op}}$.
\end{proof}

Apply Lemma~\ref{lem:cporder} to the unital completely positive map $\Psi_t$ whose GNS implementation is $Q_t$ and restrict the argument to $j(\Gamma_q)$. 
Compressing to the space $\cH_\delta$ preserves the complete-positive-order estimate, since $P_\delta$ commutes with both module actions $\pi_l, \pi_r$; 
discarding its orthogonal complement adds a completely positive map.
Therefore, we obtain
\begin{align}\label{eq:wirthhyp}
 \vec T_t^*\pi_l^{\cH_\delta}(\cdot)\vec T_t
 &\le_{\mathrm{cp}}e^{-2t}\pi_l^{\cH_\delta}\circ\Phi_t,\notag\\
 \vec T_t^*\pi_r^{\cH_\delta}(\cdot)\vec T_t
 &\le_{\mathrm{cp}}e^{-2t}\pi_r^{\cH_\delta}\circ\Phi_t.
\end{align}


\begin{theorem}
Suppose that $\{\Phi_t\}_{t\geq 0}$ is a $q$-Araki-Woods Ornstein-Uhlenbeck semigroup on $\Gamma_q(\cH_{\RR},U)$ with no boundedness assumption on $A$. 
Then the semigroup satisfies $\cfm(2)$ and $\cmlsi(2)$.
In particular, for a $\sigma$-finite ancillary von Neumann algebra $\cB$ and any normal state $\rho$ on $\Gamma_q\,\overline\otimes\,\cB$,
\begin{equation}\label{eq:completeentropy}
 D\!\left((\Phi_{t}^*\otimes\Id)(\rho)\,\middle\Vert\,\varphi\otimes\rho_{\cB}\right)
 \le e^{-2t}D(\rho\Vert\varphi\otimes\rho_{\cB}).
\end{equation}
Infinite relative entropy is allowed on the right. 
If $\cH_{\RR}\ne\{0\}$, the decay rate $2$ and the scalar constant $\alpha_1=1/2$ are optimal.
\end{theorem}

\begin{proof}
We have verified $\Gamma$-regularity, the closed first-order calculus, intertwining, and both complete-positivity-order bounds~\eqref{eq:wirthhyp}. 
Wirth's intertwining theorem~\cite[Theorem 3.5]{Wir25} therefore applies with $K=1$, and yields $\cfm(2)$ and $\cmlsi(2)$.
This shows that the entropy decay rate $\kappa\geq 2$.
Its entropy equivalence theorem~\cite[Theorem 2.7]{Wir25} gives exponential decay. The fixed algebra of $\Phi_t$ is scalar: the zero eigenspace of the number operator is $\CC\Omega$, and the vacuum is separating. 
After amplification the fixed algebra is $1\otimes \cB$, and the conditional expectation has predual $\rho\mapsto\varphi\otimes\rho_{\cB}$. 
This proves exactly~\eqref{eq:completeentropy}. 
The complete inequalities permit arbitrary ancillary states, not just product inputs. 
Non-faithful marginals can be handled on their support or by a faithful approximation.

For the sharpness, we choose a nonzero centered self-adjoint first-chaos element $a$ with compact modular spectral support, and define states
\[
 D_{\rho_\varepsilon}=D_\varphi^{1/2}(1+\varepsilon a)D_\varphi^{1/2}.
\]
They are positive and normalized for sufficiently small real $\varepsilon$. 
The GNS symmetry implies $\Phi_{t}^*(\rho_\varepsilon)=\rho_{e^{-t}\varepsilon}$.
By Lemma \ref{lem:entTaylor}, we see that the relative entropy has the following expansion,
\[
 D(\rho_\varepsilon\Vert\varphi)=\frac{1}{2}\varepsilon^2 \ip\xi{\vartheta(\log\Delta_\varphi)\xi}+o(\varepsilon^2),
\]
where $\xi=\Delta_\varphi^{1/4}a\Omega$.
The bounded spectral kernel is strictly positive at every finite frequency, so the coefficient is finite and nonzero. 
An entropy decay rate $\kappa$ would force $e^{-2t}\le e^{-\kappa t}$, hence $\kappa\le2$. 
This proves the constants in the inequalities are optimal.

In the normalization of~\eqref{qubit:energy}, the limiting form is $\mathcal E_{1, \varphi}=\mathcal I_{\cL}/4$, with $\mathcal I_{\cL}$ the entropy production, so the optimal scalar $\alpha_1$ equals $1/2$.
\end{proof}


\section{Reverse Hypercontractivity}
In this section, we shall obtain the reverse hypercontractivity for $q$-Araki-Woods Ornstein-Uhlenbeck semigroups.

For $0<p<1$ and $x\ge0$ use the same positive Haagerup expression~\eqref{main:norm}; it is a quasi-norm. 
For $x \geq \varepsilon 1$ for $\varepsilon>0$ and $p<0$, we define
\begin{align*}
    \|x\|_{p, \varphi}=\|x^{-1}\|_{-p, \varphi}^{-1}.
\end{align*}
We do not define a single two-sided value at $p=0$.


Now we shall apply a similar argument in Section \ref{sec:cmlsi} to show that the Ornstein-Uhlenbeck semigroup on the baby-Fock model has $\cmlsi(2)$ and the modified logarithmic Sobolev constant $\displaystyle \alpha_1(\cL_n^{\varepsilon})\geq \frac{1}{2}$ for all $n$.

Normalize a spin generator by $\widetilde{\gamma}_i=(\mu_i^2+\mu_i^{-2})^{-1/2}\gamma_i$, so $\widetilde{\gamma}_i^2=0$ and $\widetilde{\gamma}_i^*\widetilde{\gamma}_i+\widetilde{\gamma}_i\widetilde{\gamma}_i^*=1$. 
Double each underlying mode to $(i,0)$ and $(i,1)$. 
Use sign $-1$ between the two copies of the same mode, and sign $\varepsilon_{ij}$ between every pair of copies of distinct underlying modes. 
Give the two copies identical thermal parameters. 
The map
\begin{align*}
 \widetilde{\gamma}_{i,0}&\longmapsto\cos\theta\,\widetilde{\gamma}_{i,0}+\sin\theta\,\widetilde{\gamma}_{i,1},\\
\widetilde{\gamma}_{i,1}&\longmapsto-\sin\theta\,\widetilde{\gamma}_{i,0}+\cos\theta\,\widetilde{\gamma}_{i,1}
\end{align*}
preserves the defining relations. 
For distinct underlying modes, bilinearity and the identical signs for all copy combinations preserve every mixed relation. These maps therefore define $*$-automorphisms.

Let $v_{i,e}=\widetilde{\gamma}_{i,e}^*\widetilde{\gamma}_{i,e}$, $e=0,1$. 
The doubled thermal density is, up to normalization,
\[
 \prod_i (a_i/b_i)^{v_{i,0}+v_{i,1}}.
\]
The rotation preserves every sum $v_{i,0}+v_{i,1}$, hence preserves the state. 
For $r_i=\varphi(v_{i,e})=a_i$,
\begin{align*}
 \alpha_\theta(v_{i,0}-r_i)
 &=\cos^2\theta(v_{i,0}-r_i)+\sin^2\theta(v_{i,1}-r_i)\\
 &\quad+\sin\theta\cos\theta(\widetilde{\gamma}_{i,0}^*\widetilde{\gamma}_{i,1}+\widetilde{\gamma}_{i,1}^*\widetilde{\gamma}_{i,0}).
\end{align*}
It follows that compression of this rotation multiplies a degree-$k$ basis word by $(\cos\theta)^k$. Its derivative has exactly one copy-$1$ odd factor and copy-$0$ degree $k-1$.

The auxiliary map damping only the copy-$0$ degrees is completely positive. 
To see this without assuming a multivariable tensorization result, put any individual mode last in the occupation realization. 
Its degree damping is exactly its thermal qubit channel tensored with the identity on the remaining algebra. 
Such single-mode maps are unital completely positive, state-preserving, GNS-symmetric and diagonal in the common ordered basis. 
They commute. 
Their product over the copy-$0$ modes is the required auxiliary semigroup.

The rotation norm identity, intertwining, and Stinespring comparison from Section~\ref{sec:cmlsi} now apply in finite dimensions.
They give, uniformly in signs, mode number, and thermal weights,
\begin{equation}\label{eq:spinMLSI}
 \cmlsi(2)\quad\text{and in particular}\quad \alpha_1(\cL_n^{\varepsilon})\ge\frac12.
\end{equation}

\begin{proposition}\label{prop:finiteReverse}
Suppose that $\{\Phi_t^{(n)}\}_{t\geq 0}$ is the Ornstein-Uhlenbeck semigroup on the baby-Fock model.
Then 
\begin{equation}\label{eq:reverse1}
 \norm{\Phi_t^{(n)} (X)}_{p, \tau_n^{\varepsilon}}\ge\norm X_{r,\tau_n^{\varepsilon}}
 \quad\text{whenever}\quad
 t\ge\frac12\log\frac{1-p}{1-r},
\end{equation}
in each of the following cases:
\begin{enumerate}
 \item $0<p<r<1$, with $X\ge0$;
 \item $-\infty<p<r<0$, with invertible $X\ge 0$.
\end{enumerate}    
\end{proposition}
\begin{proof}
    In a finite model the definitions of entropy and energy extend to $0<s<1$ using the same positive-definite matrix formulas, with $s'=s/(s-1)<0$. The finite-dimensional Stroock--Varopoulos inequality~\cite[Theorem 14 and Corollary 16]{BDR}, together with~\eqref{eq:spinMLSI}, gives
\[
 \mathcal E_{s, \tau_n^{\varepsilon}}(X)\ge\frac12\Ent_{s, \tau_n^{\varepsilon}}(X),\qquad 0<s<1,
\]
for positive invertible $X$. 

Fix $0<p<r<1$ and let $s(t)=1-(1-r)e^{2t}$. Until $s(t)=p$ this exponent remains strictly positive. The matrix differentiation formulas give
\begin{align*}
 \frac d{dt}\log\norm{\Phi_t^{(n)}(X)}_{s(t)}
 &=\frac{\dot s\Ent_{s, \tau_n^{\varepsilon}}(\Phi_t^{(n)}(X))-4(s-1)\mathcal E_{s, \tau_n^{\varepsilon}}(\Phi_t^{(n)}(X))}{s^2\norm{\Phi_t^{(n)}(X)}_s^s}\\
 &=\frac{4(1-s)(\mathcal E_{s, \tau_n^{\varepsilon}}(\Phi_t^{(n)}(X))-\frac12\Ent_{s, \tau_n^{\varepsilon}}(\Phi_t^{(n)}(X)))}{s^2\norm{\Phi_t^{(n)}(X)}_s^s}\ge0.
\end{align*}
At the terminal time we have
\begin{equation}\label{eq:positiveReverse}
 \norm{\Phi_t^{(n)}(X)}_p\ge\norm X_r,\qquad
 t=\frac12\log\frac{1-p}{1-r}.
\end{equation}
The same energy comparison gives reverse contractivity at each fixed $0<s<1$, so all larger times also work. 
Approximating $X$ by $X+\eta 1$ proves the inequality for $X\ge0$.

If $-\infty<p<r<0$, their conjugates satisfy $0<r'<p'<1$. Weighted reverse H\"older duality gives, for positive-definite matrices $Z$,
\[
 \norm Z_{p,\sigma}
 =\inf_{Y>0}\frac{\langle Z,Y\rangle_\sigma}{\norm Y_{p',\sigma}},
 \qquad \langle X,Y\rangle_\sigma=\Tr(X\sigma^{1/2}Y\sigma^{1/2}).
\]
This identity follows from the ordinary Schatten reverse H\"older variational formula after the symmetric weight conjugations; the exponents match because $1/p+1/p'=1$. 
KMS symmetry and Inequality~\eqref{eq:positiveReverse}, now applied from exponent $p'$ to $r'$, yield
\begin{align*}
 \norm{\Phi_t^{(n)}(X)}_{p,\sigma}
 &=\inf_{Y>0}\frac{\langle X,\Phi_t^{(n)}(Y)\rangle_\sigma}{\norm Y_{p',\sigma}}\\
 &\ge\norm X_{r,\sigma}
       \inf_{Y>0}\frac{\norm{\Phi_t^{(n)}(Y)}_{r',\sigma}}{\norm Y_{p',\sigma}}
 \ge\norm X_{r,\sigma}.
\end{align*}
The time is unchanged because $\displaystyle \frac{1-r'}{1-p'}=\frac{1-p}{1-r}$. 
This proof stays away from zero throughout.
\end{proof}

Raynaud's compatible ultraproduct theorem~\cite[Theorems 3.1, 3.6 and 5.1]{Raynaud} applies to every $0<s<\infty$, 
preserves the positive cones and bimodule action, and therefore applies to symmetric positive quasi-norms. 
Use the bounded positive representatives compatible with Nou's state-support embedding and its Ornstein-Uhlenbeck semigroup intertwining. 
Then both the input and the evolved quasi-norms converge along the same ultrafilter. 

\begin{theorem}\label{thm:reverse}
Suppose that $\{\Phi_t\}_{t\geq 0}$ is a $q$-Araki-Woods Ornstein-Uhlenbeck semigroup with no boundedness assumption on $A$.
Then
\begin{equation}\label{eq:reverse}
 \norm{\Phi_t (x)}_{p,\varphi}\ge\norm x_{r,\varphi}
 \quad\text{whenever}\quad
 t\ge\frac12\log\frac{1-p}{1-r},
\end{equation}
in each of the following cases:
\begin{enumerate}
 \item $0<p<r<1$, with bounded $x\ge0$;
 \item $-\infty<p<r<0$, with bounded $x\ge\varepsilon1$ for some $\varepsilon>0$.
\end{enumerate}
\end{theorem}
\begin{proof}
We use the Groh--Raynaud ultraproduct $\mathcal R$, state $\omega$, support $e$, embedding $\iota$, and maps $\mathcal Q_t$ constructed in
Section~\ref{spin:CLTsection}. 
In particular $\mathcal Q_t$ preserves $\{e\}'$, its corner map intertwines $\Phi_t$, and \eqref{eq:raynaudnorm} applies at every positive exponent.

It suffices initially to take $X=P^*P+\epsilon1$, with $P$ a finite polynomial and $\epsilon\geq0$.  
Let
$X_m=P(s_m)^*P(s_m)+\epsilon 1$ be the raw positive approximants and
choose $R>\|X\|$.  
Put $Z_m=f_R(X_m)$, where
$f_R(u)=\min(u,R)$ for $u\geq0$.  
For every integer $k>2$,
\[
 (u-R)_+^2\leq R^{2-k}u^k,\qquad u\geq0,
\]
and hence
\[
 \tau_m^{\varepsilon}((X_m-Z_m)^2)
 \leq R^{2-k}\tau_m^{\varepsilon}(X_m^k).
\]
All these moments converge.
Taking the ultralimit and then $k\to\infty$ gives
\[
 \lim_{m,\mathcal U}\|(X_m-Z_m)D_m^{1/2}\|_2=0,
\]
since $R^{2-k}\varphi(X^k)\leq
R^2(\|X\|/R)^k\to0$.

Write
\[
\widetilde s_{i,m}=h_C(\Re s_{i,m})+i\,h_C(\Im s_{i,m}),
\qquad
h_C(u)=u\,\mathbf{1}_{(-C,C)}(u),
\]
with \(C\) greater than the norms of the limiting self-adjoint coordinates.
Comparing $X_m$ also with the uniformly bounded positive polynomial $Y_m=P(\widetilde{s}_m)^*P(\widetilde{s}_m)+\epsilon 1$ in Nou's cutoff Gaussian generators, and applying the proof of \cite[Lemma 4.3]{LR},
we obtain that $\displaystyle \lim_{m,\mathcal U}\|(Y_m-Z_m)D_m^{1/2}\|_2=0$.
Thus $a=[Z_m]_{\mathcal U}$ has the same support compression as that bounded polynomial lift, and commutes with $e$: 
their self-adjoint difference $w$ satisfies $\omega(w^2)=0$, hence $we=ew=0$.
It follows that $eae=\iota(X)$.  
Moreover
$0\leq Z_m\leq R1$, and $Z_m\geq \epsilon 1$.
For $0<p<r<1$, apply \eqref{eq:raynaudnorm} directly to $Z_m$ and $\Phi_t^{(m)}(Z_m)$ and pass Proposition \ref{prop:finiteReverse} to the ultralimit.
This also applies when $\epsilon=0$.


For $\epsilon>0$, the preceding lifts satisfy
\[
 \epsilon1\leq Z_m\leq R 1,
 \qquad
 \epsilon1\leq \Phi_t^{(m)}(Z_m)\leq R 1.
\]
Because $e$ commutes with both $a=[Z_m]_{\mathcal U}$ and $\mathcal Q_t(a)$,
continuous functional calculus on $[\epsilon,R]$ gives exactly
\[
 e[Z_m^{-1}]_{\mathcal U} e=\iota(X^{-1}),
 \qquad
 e[(\Phi_t^{(m)}(Z_m))^{-1}]_{\mathcal U}e
          =\iota((\Phi_t(X))^{-1}).
\]
The inverses are uniformly bounded by $\epsilon^{-1}$.
By Equation \eqref{eq:raynaudnorm} and Proposition \ref{prop:finiteReverse}, we have that
\[
 \|\Phi_t(X)\|_{p,\varphi}
 =\lim_{m,\mathcal U}\|\Phi_t^{(m)}(Z_m)\|_{p,\tau_m^{\varepsilon}}
 \ge\lim_{m,\mathcal U}\|Z_m\|_{r,\tau_m^{\varepsilon}}
 =\|X\|_{r,\varphi},
\]
for $-\infty<p<r<0$ and $0<p<r<1$.

Finally, for the positive range $0<p<r<1$ and $x\geq 0$, we choose uniformly bounded polynomials $P_j\to x^{1/2}$ strongly$^*$ and put $x_j=P_j^*P_j$.
For $x\ge \epsilon 1$ in the negative range $-\infty<p<r<0$, instead choose uniformly bounded $P_j\to(x-\epsilon 1)^{1/2}$ strongly$^*$ and set $x_j=P_j^*P_j+\epsilon 1$. 
Then $x_j\ge \epsilon 1$ and $x_j\to x$ strongly$^*$.  
Normality of the semigroup and the Schwarz inequality imply $\Phi_t(x_j)\to \Phi_t(x)$ strongly$*$.  
Bounded strong-$*$ convergence implies convergence in every symmetric weighted $L^s(\cM)$ for $0<s<\infty$; 
for $0<s<2$, we use the generalized H\"older inequality with respect to $\varphi$ and for $s\geq2$ use interpolation between $L^2(\cM)$ and $\cM$.  
Inversion is strong-$*$ continuous on sets uniformly bounded below by $\epsilon 1$.  
This extends the inequalities to exactly the bounded inputs stated in Theorem~\ref{thm:reverse}.

For infinitely many almost-periodic modes use the increasing
expected finite-mode algebras.  In the general representation use
Nou's spectral discretization and the same expected-embedding
argument; degree is preserved, so the number semigroups intertwine.
If $\log A$ is unbounded, first restrict to the real invariant
spectral subspaces for $[-R,R]$.  Their second quantizations are
state-preserving expectations, the corresponding bounded inputs
converge strongly as $R\to\infty$, and the constants above are
independent of $R$.

\end{proof}

\begin{remark}
The same universal reverse hypercontractive time curve does not extend to all \(p<0<r<1\).
Consider a one-mode almost-periodic $q$-Araki–Woods representation with modular eigenvalue $\lambda= 4$ and let $g$ be its generalized circular generator with $D_\varphi g =\lambda g D_\varphi$.
Let 
\begin{align*}
    x_\epsilon =1+\epsilon (g+g^*).
\end{align*}
For $\epsilon$ small enough and $x_{\epsilon}>0$, we have that $\Phi_t(x_{\epsilon})=1+\epsilon e^{-t}(g+g^*)$.
The second-order expansions are
\begin{align*}
    \|x_{\epsilon}\|_{1/2, \varphi}=&1-\frac{\lambda}{1+\lambda} \varphi(g^*g) \epsilon^2 +O(\epsilon^3), \\
    \|\Phi_t(x_\epsilon)\|_{-1, \varphi} = & 1-(1+\lambda) \varphi(g^*g) e^{-2t}\epsilon^2 +O(\epsilon^3).
\end{align*}
Then the reverse hypercontractivity from the input exponent $r=\frac{1}{2}$ to the output exponent $p=-1$ implies that 
\begin{align*}
    e^{-2t} \leq \frac{\lambda}{(1+\lambda)^2}.
\end{align*}
By taking $\lambda=4$ and $t=\log 2$, we have that
\begin{align*}
    \|\Phi_{\log 2} (x_\epsilon)\|_{-1, \varphi} =1-\frac{5}{4}\varphi(g^*g)\epsilon^2 +O(\epsilon^3)< 1-\frac{4}{5} \varphi(g^*g)\epsilon^2+O(\epsilon^3) =\|x_\epsilon\|_{1/2, \varphi}.
\end{align*}
This gives a contradiction.
\end{remark}

\begin{remark}\label{rem:crosszero}
The same reverse curve does not extend to all $p<0<r<1$ even for a qubit.
Let $\sigma=\operatorname{diag}(a,b)$, $a>b>0$, $a+b=1$, and
\[
 Q=\begin{pmatrix}0&1\\1&0\end{pmatrix},\qquad
 X_\varepsilon=I+\varepsilon Q,\qquad 0<|\varepsilon|<1.
\]
For the thermal qubit semigroup, $\Theta_tX_\varepsilon=I+e^{-t}\varepsilon Q$. Direct inversion gives
\begin{equation}\label{eq:minus1exact}
 \norm{\Theta_tX_\varepsilon}_{-1,\sigma}
 =\bigl(\Tr(\sigma(\Theta_tX_\varepsilon)^{-1})\bigr)^{-1}
 =1-e^{-2t}\varepsilon^2.
\end{equation}
For $s=1/2$, use $\bigl(\Tr\sqrt H\bigr)^2=\Tr H+2\sqrt{\det H}$ for a positive $2\times2$ matrix. It gives
\begin{equation}\label{eq:halfexact}
 \norm{X_\varepsilon}_{1/2,\sigma}
 =a^2+b^2+2ab\sqrt{1-\varepsilon^2}
 =1-ab\varepsilon^2+O(\varepsilon^4).
\end{equation}
Thus a reverse estimate from $r=1/2$ to $p=-1$ requires $e^{-2t}\le ab$.

Take $a=4/5$, $b=1/5$. The proposed universal classical time would be
$t=\frac12\log((1-(-1))/(1-1/2))=\log2$. But at that time
\begin{align*}
 \norm{\Theta_{\log2}X_\varepsilon}_{-1,\sigma}
 &=1-\frac14\varepsilon^2,\\
 \norm{X_\varepsilon}_{1/2,\sigma}
 &=\frac{17}{25}+\frac8{25}\sqrt{1-\varepsilon^2}
 =1-\frac4{25}\varepsilon^2+O(\varepsilon^4).
\end{align*}
The first expression is strictly smaller for sufficiently small nonzero $\varepsilon$, whereas reverse hypercontractivity would require it to be larger. 
\end{remark}

\begin{remark}
The crossing-zero range in the reverse hypercontractivity statement of \cite[Theorem 11 and Corollary 17]{BDR}, is not valid for the symmetric weighted norms used here. 
The obstruction is the possible difference between their one-sided limits at exponent zero. 
Our argument uses the Stroock–Varopoulos comparison only on positive exponents and treats negative exponents by duality.
\end{remark}

\section*{Declaration of generative AI use}

During the preparation of this work, the authors used generative AI tools to improve the language and readability of the manuscript and to assist with developing and checking mathematical arguments.
The authors reviewed and revised the AI-assisted content and verified the mathematical statements and proofs.
The authors take full responsibility for the accuracy and integrity of the manuscript.

\bibliographystyle{plain}
\bibliography{araki}
\end{document}